\documentclass[reqno,a4paper]{amsart}
\usepackage{amssymb}
\usepackage{amsmath}
\newcommand{\half}{\frac{1}{2}}

\begin{document} 
\newtheorem{prop}{Proposition}[section]
\newtheorem{Def}{Definition}[section]
\newtheorem{theorem}{Theorem}[section]
\newtheorem{lemma}{Lemma}[section]
 \newtheorem{Cor}{Corollary}[section]

\title[Klein-Gordon-Zakharov]{\bf Local well-posedness for the Klein-Gordon-Zakharov system in 3D}
\author[Hartmut Pecher]{
{\bf Hartmut Pecher}\\
Fakult\"at f\"ur  Mathematik und Naturwissenschaften\\
Bergische Universit\"at Wuppertal\\
Gau{\ss}str.  20\\
42119 Wuppertal\\
Germany\\
e-mail {\tt pecher@math.uni-wuppertal.de}}
\date{}

\begin{abstract}
	We study the Cauchy problem for the Klein-Gordon-Zakharov system in 3D with low regularity data. We lower down the regularity to the critical value with respect to scaling up to the endpoint. The decisive bilinear estimates are proved by means of methods developed by Bejenaru-Herr for the Zakharov system and already applied by Kinoshita to the Klein-Gordon-Zakharov system in 2D.

\end{abstract}
\maketitle
\renewcommand{\thefootnote}{\fnsymbol{footnote}}
\footnotetext{\hspace{-1.5em}{\it 2010 Mathematics Subject Classification:} 
35Q55, 35A01 }
\normalsize 
\setcounter{section}{0}

\section{Introduction}
Consider the Cauchy problem for the Klein-Gordon-Zakharov system in three space dimensions:
\begin{align}
\label{1.1}
\partial_t^2 u - \Delta u + u &= -nu \\
\nonumber \partial_t^2 n - c^2 \Delta n & = \Delta(|u|^2)
\end{align}
on condition that $ 0 < c < 1$ with Cauchy data
\begin{equation}
\label{1.1'}
u(0) = u_0 \, , \, (\partial_t u)(0) = u_1 \, , \, n(0) = n_0 \, , \, (\partial_t n)(0)= n_1 \, .
\end{equation}
Here $u$ and $n$ are real-valued functions.
Assume that the data belong to the following standard Sobolev spaces:
\begin{equation}
\label{1.1''}
 u_0 \in H^{s+1} \, , \, u_1 \in H^s \, , \, n_0 \in H^s \, , \, n_1 \in H^{s-1} \, . 
 \end{equation}
It is standard to transform the system into a first order (in t) system. \\
Let $A:= -\Delta + 1$ .
Define $u_{\pm} = u \pm i A^{-\half} \partial_t u $ , $ n_{\pm} = n \pm i c^{-1} A^{-\half} n $ , so that $u=\half(u_++u_-)$ , $ n =\half( n_++n_-)$ . Then the system (\ref{1.1}) is equivalent to
\begin{align}
\label{1.2}
i \partial_t u_{\pm} \mp A^{\half} u_{\pm} & = \pm  \frac{1}{4} A^{-\half} ((n_++n_-)(u_++u_-)) \\
\nonumber
i \partial_t n_{\pm} \mp c A^{\half} n_{\pm} & = \mp \frac{1}{4c} \Delta A^{-\half} |u_++u_-|^2 \mp \frac{c}{2} A^{-\half} (n_++n_-) \, .
\end{align}  The transformed Cauchy data are 
$$u_{\pm_0} = u_0 \pm iA^{-\half}u_1 \in H^{s+1} \, , \, n_{\pm_0} = n_0 \pm i c^{-1}A^{-\half} n_1 \in H^s \, . $$

The Klein–Gordon–Zakharov system describes the interaction between
Langmuir waves and ion sound waves in a plasma. In this application the constant $c$ actually fulfills $ c < 1$ (see Masmoudi-Nakanishi \cite{MN}).

We are interested in low regularity well.-posedness results.

In 2D Tsugawa proved local well-posedness  for
$ s \ge - \half$ . Also in 2D S. Kinoshita \cite{K} substantially improved this result to $ s > - \frac{3}{4}$ . Crucial for this improvement were bilinear estimates given  by Bejenaru-Herr-Holmer-Tataru \cite{BHHT} in order to prove an optimal local well-posedness result for the related Zakharov system in 2D based on a result by Bejenaru-Herr-Tataru \cite{BHT}. Kinoshita was also able to prove that this result is optimal up to the endpoint. 

I. Kato [8]  proved that (\ref{1.1}) is
locally well-posed at $s = \frac{1}{4}$, when the space dimension $ d = 4 $ , and $s = s_c + \frac{1}{d + 1}$, when $d \ge 5$, where
$s_c =  \frac{d}{2}-2$ is the critical exponent of (\ref{1.1})with respect to scaling.

The assumption $0<c<1$ is crucial in all these results.

In 3D, Ozawa, Tsutaya and Tsutsumi \cite{OTT}  proved that (\ref{1.1}) is globally well-posed in
the energy space $H^1 \times L^2 \times L^2 \times {\dot H}^{-1}$. It was very important for their argument to assume different propagations speeds ($c \neq 1$).
As far as I know there were no well-posedness results for data with $s < 0$ , although  the critical exponent with respect to scaling is $s=-\half$ . 
This is easily seen by considering the rescaling $u_{\lambda}(x,t)= \lambda u(\lambda x, \lambda t)$ , $n_{\lambda}(x,t)= \lambda^2 n(\lambda x, \lambda t)$ for the system ignoring the linear term $u$ in (\ref{1.1}), which plays no role as long local results are concerned.
The aim of the present paper is to lower down the regularity assumptions on the data and to close this gap as far as possible in the case $0<c<1$ . In fact the main result shows local well-posedness for $s >-\half$ , thus leaving open only the critical case $s=-\half$ . The proof combines the method used by Bejenaru-Herr \cite{BH} for their optimal well-posedness result for the 3D Zakharov system and  Kinoshita`s approach for the optimal well-posedness result for the 2D Klein-Gordon-Zakharov system \cite{K}. Bejenaru-Herr introduced a suitable additional decomposition with respect to angular variables in frequency space, which also  plays a fundamental role both in Kinoshita's article and in our paper.

We define the solution spaces of Bourgain - Klainerman - Machedon type as
$$X^{s,b}_{\pm} = \{ u \in \mathcal{S}'(\mathbb{R}^4) : \|u\|_{X^{s,b}_{\pm}} < \infty \} \, , $$
where $$\|u\|_{X^{s,b}_{\pm}} = \| \langle \xi \rangle^s \langle \tau \pm |\xi| \rangle^b \tilde{u}(\xi,\tau)\|_{L^2_{\xi \tau}} $$
and $\tilde{u}$ denotes the space-time Fourier transform of $u$ , and similarly
$X^{s,b}_{\pm,c}$ with norm
$$\|n\|_{X^{s,b}_{\pm,c}} =  \| \langle \xi \rangle^s \langle \tau \pm c|\xi| \rangle^b \tilde{u}(\xi,\tau)\|_{L^2_{\xi \tau}} \, .$$
We also define $X^{s,b}_{\pm}[0,T]$ and $X^{s,b}_{\pm,c}[0,T]$ as the space of the restrictions to $[0,T] \times \mathbb{R}^4$.

Now we formulate our main theorem. 
\begin{theorem}
	\label{Theorem0.1}
	Assume $0 < c < 1 $ and $ s > -\half$ .
	The data are assumed to fulfill $u_{\pm_0} \in H^{s+1}(\mathbb{R}^3)$ , $n_{\pm_0} \in H^s(\mathbb{R}^3)$. Then there exists $ b > \frac{1}{2}$ , $T > 0 $ ,  $T=T(\|u_{\pm_0}\|_{H^{s+1}},
	\|n_{\pm_0}\|_{H^s})$ , such that the problem (\ref{1.2}) with Cauchy data $u_{\pm_0}$ and $n_{\pm_0}$ has a unique local solution
	$$ u_{\pm} \in X^{s+1,b}_{\pm}[0,T]  \, , \,
	n_{\pm} \in X^{s,b}_{\pm,c}[0,T] \, .$$
	 
\end{theorem}
An immediate consequence is 
\begin{Cor}
	\label{Cor.1}
	Assume $0<c<1$ and $s > - \half$ . Let the data fulfill (\ref{1.1''}) .
	The Cauchy problem problem (\ref{1.1}),(\ref{1.1'}) has a unique local solution
	$$u \in X^{s+1,b}_+[0,T] + X^{s+1,b}_-[0,T] \, , \, n \in X^{s,b}_+[0,T] + X^{s,b}_-[0,T] \, , $$
	where $ b > \half$ . This solution has the property
	$$ u \in C^0([0,T],H^{s+1}) \cap C^1([0,T],H^s) \, , \,  n \in C^0([0,T],H^{s}) \cap C^1([0,T],H^{s-1}) \, . $$
	\end{Cor}
{\bf Remarks:} 1. The solution depends continuously on the initial data and persistence of higher regularity holds. \\
2. For $-\half < s \le 0$ it is also possible to prove by the  same arguments local well-posedness for data 
$$
u_0 \in H^{s+1} \, , \, u_1 \in H^s \, , \, n_0 \in \dot{H}^s \, , \, n_1 \in \dot{H}^{s-1} \, .  
$$
It is namely easily verified that the bilinear estimates formulated in Proposition \ref{Theorem2.1} are sufficient to cover this case. For details we refer to \cite{K}, Remarks 1 and 2.
\\[0.5em]

The decisive bilinear estimates are formulated in the following proposition.
\begin{prop}
	\label{Theorem2.1}
	Assume $s > -\half$ and $0<c<1$ . Then there exists $b \in (\half,1)$ such that
	\begin{align}
	\label{1.3}
	\|nu\|_{X^{s,b-1+}_{\pm_0}} & \lesssim \|n\|_{X^{s,b}_{\pm_1,c}} \|u\|_{X^{s+1,b}_{\pm_2}} \, , \\
	\label{1.4}
	\|u_1 u_2\|_{X^{s+1,b-1+}_{\pm_0,c}} & \lesssim \|u_1\|_{X^{s+1,b}_{\pm_1}} \|u_2\|_{X^{s+1,b}_{\pm_2}}\, ,
	\end{align}
	where $\pm_0$ , $\pm_1$ and $\pm_2$ denote independent signs.
\end{prop}
\noindent{\bf Remarks:} 1. Here and below $a\pm$ denotes $a\pm \epsilon$ for a sufficiently small $\epsilon > 0$ . \\
2. It suffices to consider $ s > -\half$ close to $-\half$ , because the other cases are implied by an application of the fractional Leibniz rule.\\[0.2em]

By duality these estimates reduce to
\begin{align}
\label{1.5}
\left|\int f g_1 g_2 \, dx dt \right| &\lesssim \|f\|_{X^{s,b}_{\pm_0,c}} \|g_1\|_{X^{s+1,b}_{\pm_1}} \|g_2\|_{X^{-s,1-b-}_{\pm_2}} \, , \\
\label{1.6}
\left|\int f g_1 g_2 \, dx dt\right| &\lesssim \|f\|_{X^{-s-1,1-b-}_{\pm_0,c}} \|g_1\|_{X^{s+1,b}_{\pm_1}} \|g_2\|_{X^{s+1,b}_{\pm_2}} \, .
\end{align}
\\[0.2em]
For dyadic numbers $N$ and $L$ we define
\begin{align*}
K_{N,L}^{\pm} &= \{ (\xi,\tau) \in \mathbb{R}^4 : N \le \langle \xi \rangle \le 2N\, , \, L \le \langle \tau \pm |\xi| \rangle \le 2L \}\, , \\
K_{N,L}^{\pm,c} &= \{ (\xi,\tau) \in \mathbb{R}^4 : N \le \langle \xi \rangle \le 2N \, , \, L \le \langle \tau \pm c|\xi| \rangle \le 2L \}
\end{align*}
and
$$ P_{K^{\pm}_{N,L}} = \mathcal{F}_{(x,t)}^{-1} \chi_{K^{\pm}_{N,L}} \mathcal{F}_{(x,t)} \quad , \quad P_{K^{\pm,c}_{N,L}} = \mathcal{F}_{(x,t)}^{-1} \chi_{K^{\pm,c}_{N,L}} \mathcal{F}_{(x,t)} $$
as well as $P_N = \mathcal{F}_{(x,t)}^{-1}  \chi_{\{ N \le \langle \xi \rangle \le 2N\}} \mathcal{F}_{(x,t)}$ .
\\[0.5em] 
{\bf Acknowledgement:} Special thanks go to Shinya Kinoshita, who kindly helped me to recognize a number of details in his paper. I also thank Axel Gr\"unrock for several very helpful contributions.

\section{The case $\pm_1 \neq \pm_2$}

\begin{lemma}
	\label{Lemma1.1}
	Let $\tau=\tau_1+\tau_2$ , $\xi=\xi_1+\xi_2$ and $0 < c < 1$ . Then we have
	$$\max(\langle \tau \pm_0 c |\xi| \rangle , \langle \tau_1 - |\xi_1| \rangle , \langle \tau_2 - |\xi_2|\rangle ) \gtrsim \max(|\xi_1|,\xi_2|) \, . $$
\end{lemma}
\begin{proof} We easily obtain
	 \begin{align*}
	 &\max(\langle \tau \pm_0 c |\xi| \rangle , \langle \tau_1 - |\xi_1| \rangle , \langle \tau_2 - |\xi_2|\rangle ) \\
	 & \ge |\tau \pm_0 c |\xi|-(\tau_1-|\xi_1|) - (\tau_2-|\xi_2|)| \\
	 & \ge ||\xi_1|+|\xi_2| - c |\xi|| \ge |\xi_1|+|\xi_2| -c(|\xi_1|+|\xi_2|) \\
	 & = (1-c)(|\xi_1| + |\xi_2|) \, .
	 \end{align*}
	\end{proof}
A similar result also holds in the case $\pm_1 \neq \pm_2$ , provided $|\xi_1| \nsim |\xi_2|$ .
\begin{lemma}
	\label{Lemma1.2}
	Let $\tau=\tau_1+\tau_2$ , $\xi=\xi_1+\xi_2$ , $|\xi_2| \le \frac{1-c}{2(1+c)} |\xi_1|$ or  $|\xi_1| \le \frac{1-c}{2(1+c)} |\xi_2|$ , and $0 < c < 1$ . Then the following estimate applies:
	$$\max(\langle \tau \pm_0 c |\xi| \rangle , \langle \tau_1 - |\xi_1| \rangle , \langle \tau_2 + |\xi_2|\rangle ) \ge \frac{1-c}{2} \max(|\xi_1|,\xi_2|) \, . $$
\end{lemma}
\begin{proof}
	In the case   $|\xi_2| \le \frac{1-c}{2(1+c)} |\xi_1|$ we obtain
	\begin{align*}  
	&\max(\langle \tau \pm_0 c |\xi| \rangle , \langle \tau_1 - |\xi_1| \rangle , \langle \tau_2 + |\xi_2|\rangle ) \ge |\pm_0 c|\xi|+|\xi_1|-|\xi_2|| \\
	& \ge ||\xi_1|-|\xi_2|| - c(|\xi_1|+|\xi_2|) \ge (1-\frac{1-c}{2(1+c)} - c - \frac{c(1-c)}{2(1+c)}) |\xi_1| = \frac{1-c}{2} |\xi_1| \, .
	\end{align*}
	The other case is treated similarly.
	\end{proof}

Thus, if either {\bf a.}  $\pm_1 = \pm_2$ or  else {\bf b.} $\pm_1 \neq \pm_2$ and $|\xi_1| \ll |\xi_2|$ or $|\xi_2| \gg |\xi_1|$ in the sense of Lemma \ref{Lemma1.2} the estimate
\begin{equation}
\label{L_max}
L_{max} := \max(L_0,L_1,L_2) \gtrsim \max(|\xi_1|,|\xi_2|)
\end{equation}
 is true, where
$L_0 \sim \langle \tau \pm_0 c|\xi| \rangle$ , $L_1 \sim \langle \tau_1 \pm_1 |\xi_1| \rangle$ , $L_2 \sim \langle \tau_2 \pm_2 |\xi_2| \rangle$ .\\[0.2em]

In this chapter we prove the desired bilinear estimates provided (\ref{L_max}) is true.

We first review a bilinear estimate which applies regardless of signs or $c$ .
\begin{prop}
	\label{Prop.1.1}
	For any choice of signs $\pm_0,\pm_1,\pm_2$ , $ c > 0$ and $0 \le \epsilon \le \half$ the following estimate applies:
	\begin{align*}
	&\| P_{K^{\pm_0}_{N_0,L_0}} ((P_{K^{\pm_1,c}_{N_1,L_1}} f)(P_{K^{\pm_2}_{N_2,L_2}} g))\|_{L^2_{xt}} 
	\lesssim  \min(N_0,N_1,N_2)^{\half + 2\epsilon} \min(L_1,L_2)^{\half} \\ & \quad \quad \quad \quad\cdot \min(N_1,N_2)^{\half-\epsilon} \max(L_1,L_2)^{\half-\epsilon} \|P_{K^{\pm_1,c}_{N_1,L_1}} f\|_{L^2_{xt}} \|P_{K^{\pm_2}_{N_2,L_2}} g\|_{L^2_{xt}} \, .
	\end{align*}
\end{prop}
\begin{proof}
	We use \cite{S}, Theorem 1.1, which proves the estimate in the case $\epsilon =0$ and remark that the proof makes no use of a specific value of $c$. Moreover we apply the Sobolev type estimate
	$$\|P_{N_0}((P_{K^c_{N_1,L_1}} f)(P_{N_2} g))\|_{L^2_{xt}} \lesssim \min(N_0,N_1,N_2)^{\frac{3}{2}} L_1^{\half} \|P_{K^c_{N_1,L_1}} f\|_{L^2_{xt}} \|P_{N_2} g\|_{L^2_{xt}} \, .$$
	By bilinear interpolation the claimed estimate follows.
	\end{proof}
\begin{prop}
	Assume $L_{max} \gtrsim \max(N_1,N_2)$ , $0<c$ , $s > -\half$ . Define 
	$$I_1 = N_1^{-1} \int (P_{K^{\pm,c}_{N_0,L_0}} f)(P_{K^{\pm_1}_{N_1,L_1}} g_1)(P_{K^{\pm_2}_{K_{N_2,L_2}}} g_2) \, dx dt \, . $$
	Then there exists $b >\half$ such that following estimate applies:
	$$ \sum_{1 \le N_0,N_1,N_2} \sum_{L_0,L_1,L_2} I_1 \lesssim \|f\|_{X^{s,b}_{\pm,c}} \|g_1\|_{X^{s,b}_{\pm_1}} \|g_2\|_{X^{-s,1-b-}_{\pm_2}} \, . $$
\end{prop}
	\begin{proof}
{\bf Case 1:} $ 1 \le N_0 \lesssim N_1 \sim N_2$ . \\
Case 1.1: $N_1 \lesssim L_0$ . We use the notation $f= P_{K^{\pm,c}_{N_0,L_0}} f$ , $g_j = P_{K^{\pm_j}_{N_j,L_j}} g_j$ . By Prop. \ref{Prop.1.1} we obtain
\begin{align*}
&N_1^{-1} |\int f g_1 g_2 \, dx dt| \le N_1^{-1} \|f\|_{L^2} \|P_{K_{N_0,L_0}} (g_1 g_2)\|_{L^2} \\
& \lesssim N_1^{-1} N_0^{\half + 2\epsilon} N_1^{\half-\epsilon} L_1^{\half} L_2^{\half-\epsilon} N_1^{-b+} \|f\|_{X^{s,b}_{\pm,c}} \|g_1\|_{X^{s,b}_{\pm_1}} \|g_2\|_{X^{-s,1-b-}_{\pm_2}} \\
& \hspace{10em} \cdot L_0^{0-} N_0^{-s+} N_1^{-s} N_2^s L_1^{-b} L_2^{b-1+} \\
& \lesssim N_0^{\half-s+2\epsilon+} N_1^{-\half-b-\epsilon+} L_1^{\half-b} L_2^{-\half+b-\epsilon+} L_0^{0-}\|f\|_{X^{s,b}_{\pm,c}} \|g_1\|_{X^{s,b}_{\pm_1}} \|g_2\|_{X^{-s,1-b-}_{\pm_2}} \\
& \lesssim N_1^{-s-\half+} (L_0 L_1 L_2)^{0-} \|f\|_{X^{s,b}_{\pm,c}} \|g_1\|_{X^{s,b}_{\pm_1}} \|g_2\|_{X^{-s,1-b-}_{\pm_2}} \\
& \lesssim N_1^{0-} (L_0 L_1 L_2)^{0-}\|f\|_{X^{s,b}_{\pm,c}} \|g_1\|_{X^{s,b}_{\pm_1}} \|g_2\|_{X^{-s,1-b-}_{\pm_2}} \, ,
\end{align*}
if $\epsilon = b-\half+$ and $s > -\half$ .\\
Case 1.2: $N_1 \lesssim L_1$ . By Prop. \ref{Prop.1.1} we obtain
\begin{align*}
&N_1^{-1} |\int f g_1 g_2 \, dx dt| \le N_1^{-1} \|g_1\|_{L^2} \|P_{K_{N_1,L_1}} (f g_2)\|_{L^2} \\
& \lesssim N_1^{-1} N_0^{\half + 2\epsilon} N_0^{\half-\epsilon} L_0^{\half}  L_2^{\half-\epsilon} N_0^{-s}L_0^{-b} N_1^{-s} N_1^{-b+} L_1^{0-} N_2^s L_2^{b-1+} \\
& \hspace{10em} \cdot \|f\|_{X^{s,b}_{\pm,c}} \|g_1\|_{X^{s,b}_{\pm_1}} \|g_2\|_{X^{-s,1-b-}_{\pm_2}} \\
& \lesssim N_0^{1-s+\epsilon+} N_1^{-1-b+} L_0^{\half-b} L_2^{-\half+b-\epsilon+} L_1^{0-}\|f\|_{X^{s,b}_{\pm,c}} \|g_1\|_{X^{s,b}_{\pm_1}} \|g_2\|_{X^{-s,1-b-}_{\pm_2}} \\
& \lesssim N_1^{0-} (L_0 L_1 L_2)^{0-} \|f\|_{X^{s,b}_{\pm,c}} \|g_1\|_{X^{s,b}_{\pm_1}} \|g_2\|_{X^{-s,1-b-}_{\pm_2}} \, ,
\end{align*}
if $\epsilon = b-\half+$ and $s > -\half$ .	\\	
Case 1.3: $N_1 \lesssim L_2$ . Similarly we obtain
\begin{align*}
&N_1^{-1} |\int f g_1 g_2 \, dx dt| \le N_1^{-1} \|g_2\|_{L^2} \|P_{K_{N_2,L_2}} (f g_1)\|_{L^2} \\
& \lesssim N_1^{-1} N_0^{\half + 2\epsilon} N_0^{\half-\epsilon} L_0^{\half}  L_1^{\half-\epsilon} N_0^{-s}L_0^{-b} N_1^{-s}L_1^{-b}  N_2^s N_1^{b-1+} L_2^{0-} \\
& \hspace{10em} \cdot \|f\|_{X^{s,b}_{\pm,c}} \|g_1\|_{X^{s,b}_{\pm_1}} \|g_2\|_{X^{-s,1-b-}_{\pm_2}} \\
& \lesssim N_0^{1-s+\epsilon} N_1^{-2+b+} L_0^{\half-b} L_1^{\half-\epsilon-b} L_2^{0-}\|f\|_{X^{s,b}_{\pm,c}} \|g_1\|_{X^{s,b}_{\pm_1}} \|g_2\|_{X^{-s,1-b-}_{\pm_2}} \\
& \lesssim N_1^{-1+b-s+\epsilon+} (L_0 L_1 L_2)^{0-} \|f\|_{X^{s,b}_{\pm,c}} \|g_1\|_{X^{s,b}_{\pm_1}} \|g_2\|_{X^{-s,1-b-}_{\pm_2}} \, .
\end{align*}
If $ b=\half+$ , $\epsilon = b-\half+$ and $s > -\half$ . we obtain the desired bound.	\\	
{\bf Case 2:} $1 \le N_1 \lesssim N_0 \sim N_2$ .	\\
Because it can be treated similarly as Case 1 we omit it. \\
{\bf Case 3:} $1 \le N_2 \lesssim N_0 \sim N_1$ .	\\	
Case 3.1: $N_0 \lesssim L_0$ . By Prop. \ref{Prop.2.1} we obtain		
\begin{align*}
&N_1^{-1} |\int f g_1 g_2 \, dx dt| \le N_1^{-1} \|f\|_{L^2} \|P_{K_{N_0,L_0}} (g_1 g_2)\|_{L^2} \\
& \lesssim N_1^{-1} N_2^{\half + 2\epsilon} N_2^{\half-\epsilon} L_1^{\half} L_2^{\half-\epsilon} N_0^{-s}  N_0^{-b+} L_0^{0-} N_1^{-s} L_1^{-b} N_2^s L_2^{b-1+} \\
& \hspace{10em} \|f\|_{X^{s,b}_{\pm,c}} \|g_1\|_{X^{s,b}_{\pm_1}} \|g_2\|_{X^{-s,1-b-}_{\pm_2}}  \\
& \lesssim N_2^{1+s+\epsilon} N_1^{-1-2s-b+} L_1^{\half-b} L_2^{-\half+b-\epsilon+} L_0^{0-}\|f\|_{X^{s,b}_{\pm,c}} \|g_1\|_{X^{s,b}_{\pm_1}} \|g_2\|_{X^{-s,1-b-}_{\pm_2}} \\
& \lesssim N_1^{-s-\half+} (L_0 L_1 L_2)^{0-} \|f\|_{X^{s,b}_{\pm,c}} \|g_1\|_{X^{s,b}_{\pm_1}} \|g_2\|_{X^{-s,1-b-}_{\pm_2}} \
\end{align*}
for  $\epsilon = b-\half+$ . This implies the claimed estimate if $s > -\half$ . .\\		
Case 3.2: $N_0 \lesssim L_1$ . By Prop. \ref{Prop.2.1} we obtain		
\begin{align*}
&N_1^{-1} |\int f g_1 g_2 \, dx dt| \le N_1^{-1} \|g_1\|_{L^2} \|P_{K_{N_1,L_1}} (f g_2)\|_{L^2} \\
& \lesssim N_1^{-1} N_2^{\half + 2\epsilon} N_2^{\half-\epsilon} L_0^{\half} L_2^{\half-\epsilon} N_0^{-s}  N_0^{-b+} L_0^{-b} N_1^{-s} L_1^{0-} N_2^s L_2^{b-1+} \\
& \hspace{10em} \|f\|_{X^{s,b}_{\pm,c}} \|g_1\|_{X^{s,b}_{\pm_1}} \|g_2\|_{X^{-s,1-b-}_{\pm_2}}  \\
& \lesssim N_2^{1+s+\epsilon} N_1^{-1-2s-b+} L_0^{\half-b} L_1^{0-} L_2^{b-\half-\epsilon+}\|f\|_{X^{s,b}_{\pm,c}} \|g_1\|_{X^{s,b}_{\pm_1}} \|g_2\|_{X^{-s,1-b-}_{\pm_2}} \\
& \lesssim N_1^{-s-\half+} (L_0 L_1 L_2)^{0-} \|f\|_{X^{s,b}_{\pm,c}} \|g_1\|_{X^{s,b}_{\pm_1}} \|g_2\|_{X^{-s,1-b-}_{\pm_2}} \
\end{align*}
for  $\epsilon = b-\half+$ . This implies the claimed estimate if $s > -\half$ . \\		
Case 3.3: $N_0 \lesssim L_2$ . By Prop. \ref{Prop.2.1} we obtain		
\begin{align*}
&N_1^{-1} |\int f g_1 g_2 \, dx dt| \le N_1^{-1} \|g_2\|_{L^2} \|P_{K_{N_2,L_2}} (f g_1)\|_{L^2} \\
& \lesssim N_1^{-1} N_2^{\half + 2\epsilon} N_1^{\half-\epsilon} L_0^{\half} L_1^{\half-\epsilon} N_0^{-s}  N_0^{-b+} L_0^{-b} N_1^{-s} L_1^{-b} N_2^sN_0^{b-1+} L_2^{0-} \\
& \hspace{10em}\|f\|_{X^{s,b}_{\pm,c}} \|g_1\|_{X^{s,b}_{\pm_1}} \|g_2\|_{X^{-s,1-b-}_{\pm_2}}  \\
& \lesssim N_2^{\half+s+2\epsilon} N_1^{-\frac{3}{2}-2s-b-\epsilon+} L_0^{\half-b} L_1^{\half-b-\epsilon} L_2^{0-}\|f\|_{X^{s,b}_{\pm,c}} \|g_1\|_{X^{s,b}_{\pm_1}} \|g_2\|_{X^{-s,1-b-}_{\pm_2}} \\
& \lesssim N_1^{-s-b-1+\epsilon+} (L_0 L_1 L_2)^{0-} \|f\|_{X^{s,b}_{\pm,c}} \|g_1\|_{X^{s,b}_{\pm_1}} \|g_2\|_{X^{-s,1-b-}_{\pm_2}} \
\end{align*}
 This implies the claimed estimate if  $\epsilon = b-\half+$ and $s > -\half$  .\\	
 By dyadic summation the result is implied.	
 \end{proof}

 In a similar manner the following result is proved.
\begin{prop}
	Assume $L_{max} \gtrsim \max(N_1,N_2)$ , $0<c$ , $s > -\half$ . Define 
	$$I_2 =N_0 N_1^{-1} N_2^{-1} \int (P_{K^{\pm,c}_{N_0,L_0}} f)(P_{K^{\pm_1}_{N_1,L_1}} g_1)(P_{K^{\pm_2}_{K_{N_2,L_2}}} g_2) \, dx dt \, . $$
	Then there exists $b >\half$ such that following estimate applies:
	$$ \sum_{1 \le N_0,N_1,N_2} \sum_{L_0,L_1,L_2} I_2 \lesssim \|f\|_{X^{-s,1-b-}_{\pm,c}} \|g_1\|_{X^{s,b}_{\pm_1}} \|g_2\|_{X^{s,b}_{\pm_2}} \, . $$
\end{prop}
\begin{proof}
	{\bf Case 1:} $ 1 \le N_0 \lesssim N_1 \sim N_2$ . \\
	Case 1.1: $N_1 \lesssim L_0$ . By Prop. \ref{Prop.1.1} we obtain
	\begin{align*}
	&N_0 N_1^{-1} N_2^{-1} |\int f g_1 g_2 \, dx dt| \le N_0 N_1^{-1} N_2^{-1} \|f\|_{L^2} \|P_{K_{N_0,L_0}} (g_1 g_2)\|_{L^2} \\
	& \lesssim N_0 N_1^{-1}  N_2^{-1} N_0^{\half + 2\epsilon} N_1^{\half-\epsilon} L_1^{\half} L_2^{\half-\epsilon} N_1^{b-1} L_0^{0-} N_0^s N_1^{-s} L_1^{-b} N_2^{-s} L_2^{-b} \\
	& \hspace{10em} \|f\|_{X^{-s,1-b-}_{\pm,c}} \|g_1\|_{X^{s,b}_{\pm_1}} \|g_2\|_{X^{s,b}_{\pm_2}} \\
& \lesssim N_0^{\frac{3}{2}+s+2\epsilon+} N_1^{-\frac{5}{2}+b-\epsilon-2s+} L_1^{\half-b} L_2^{\half-b-\epsilon} L_0^{0-} \\
& \hspace{10em} \|f\|_{X^{-s,1-b-}_{\pm,c}} \|g_1\|_{X^{s,b}_{\pm_1}} \|g_2\|_{X^{s,b}_{\pm_2}} \\
	& \lesssim N_1^{-1-s+b+\epsilon+} (L_0 L_1 L_2)^{0-} \|f\|_{X^{-s,1-b-}_{\pm,c}} \|g_1\|_{X^{s,b}_{\pm_1}} \|g_2\|_{X^{s,b}_{\pm_2}} \\
	& \lesssim N_1^{0-} (L_0 L_1 L_2)^{0-}\|f\|_{X^{s,b}_{\pm,c}} \|g_1\|_{X^{s,b}_{\pm_1}} \|g_2\|_{X^{-s,1-b-}_{\pm_2}} \, ,
	\end{align*}
	if $\epsilon = b-\half+$ and $s > -\half$ .\\		
Case 1.2: $N_1 \lesssim L_1$ . We obtain
\begin{align*}
&N_0 N_1^{-1} N_2^{-1} |\int f g_1 g_2 \, dx dt| \le N_0 N_1^{-1} N_2^{-1} \|g_1\|_{L^2} \|P_{K_{N_1,L_1}} (f g_2)\|_{L^2} \\
& \lesssim N_0 N_1^{-1} N_2^{-1} N_1^{\half + 2\epsilon} N_0^{\half-\epsilon} L_0^{\half} L_2^{\half-\epsilon} L_0^{b-1+} N_0^s N_1^{-s-b+} L_1^{0-} N_2^{-s} L_2^{-b} \\
& \hspace{10em} \|f\|_{X^{-s,1-b-}_{\pm,c}} \|g_1\|_{X^{s,b}_{\pm_1}} \|g_2\|_{X^{s,b}_{\pm_2}} \\
& \lesssim N_0^{\frac{3}{2}+s-\epsilon+} N_1^{-\frac{3}{2}-b+2\epsilon-2s+} L_0^{b-\half-\epsilon} L_1^{0-} L_2^{\half-b-\epsilon}  \\
& \hspace{10em} \|f\|_{X^{-s,1-b-}_{\pm,c}} \|g_1\|_{X^{s,b}_{\pm_1}} \|g_2\|_{X^{s,b}_{\pm_2}} \\
& \lesssim N_1^{-s-b+\epsilon+} (L_0 L_1 L_2)^{0-} \|f\|_{X^{-s,1-b-}_{\pm,c}} \|g_1\|_{X^{s,b}_{\pm_1}} \|g_2\|_{X^{s,b}_{\pm_2}} \\
& \lesssim N_1^{0-} (L_0 L_1 L_2)^{0-}\|f\|_{X^{s,b}_{\pm,c}} \|g_1\|_{X^{s,b}_{\pm_1}} \|g_2\|_{X^{-s,1-b-}_{\pm_2}} \, ,
\end{align*}
if $\epsilon = b-\half+$ and $s > -\half$ .\\		
Case 1.3: $N_1 \lesssim L_2$ .This case is similar to Case 1.2. and therefore omitted. \\
{\bf Case 2:} $ 1 \le N_1 \lesssim N_0 \sim N_2$ . This is treated similarly. \\
{\bf Case 3:} $ 1 \le N_2 \lesssim N_0 \sim N_2$ . \\
Case 3.1: $N_0 \lesssim L_0$ .
 We obtain
\begin{align*}
&N_0 N_1^{-1} N_2^{-1} |\int f g_1 g_2 \, dx dt| \le N_0 N_1^{-1} N_2^{-1} \|f\|_{L^2} \|P_{K_{N_0,L_0}} (g_1 g_2)\|_{L^2} \\
& \lesssim N_0 N_1^{-1} N_2^{-1} N_2^{\half + 2\epsilon} N_2^{\half-\epsilon} L_1^{\half} L_2^{\half-\epsilon} L_0^{0-} N_0^s N_0^{b-1+} N_1^{-s} L_1^{-b} N_2^{-s} L_2^{-b} \\
& \hspace{10em} \|f\|_{X^{-s,1-b-}_{\pm,c}} \|g_1\|_{X^{s,b}_{\pm_1}} \|g_2\|_{X^{s,b}_{\pm_2}} \\
& \lesssim N_2^{-s+\epsilon} N_1^{b-1+} (L_0 L_1 L_2)^{0-}  \|f\|_{X^{-s,1-b-}_{\pm,c}} \|g_1\|_{X^{s,b}_{\pm_1}} \|g_2\|_{X^{s,b}_{\pm_2}} \\
& \lesssim N_1^{-s-\half+} (L_0 L_1 L_2)^{0-} \|f\|_{X^{-s,1-b-}_{\pm,c}} \|g_1\|_{X^{s,b}_{\pm_1}} \|g_2\|_{X^{s,b}_{\pm_2}}  \, ,
\end{align*}
if $\epsilon = b-\half+$ . This implies the desired bound, if $s > -\half$ .\\		
Case 3.2: $N_0 \lesssim L_1$ .
We obtain
\begin{align*}
&N_0 N_1^{-1} N_2^{-1} |\int f g_1 g_2 \, dx dt| \le N_0 N_1^{-1} N_2^{-1} \|g_1\|_{L^2} \|P_{K_{N_1,L_1}} (f g_2)\|_{L^2} \\
& \lesssim N_0 N_1^{-1} N_2^{-1} N_2^{\half + 2\epsilon} N_2^{\half-\epsilon} L_0^{\half-\epsilon} L_2^{\half}N_0^s L_0^{b-1+}  N_1^{-s} N_0^{-b+} L_1^{0-} N_2^{-s} L_2^{-b} \\
& \hspace{10em} \|f\|_{X^{-s,1-b-}_{\pm,c}} \|g_1\|_{X^{s,b}_{\pm_1}} \|g_2\|_{X^{s,b}_{\pm_2}} \\
& \lesssim N_2^{-s+\epsilon} N_1^{-b+} (L_0 L_1 L_2)^{0-} \|f\|_{X^{-s,1-b-}_{\pm,c}} \|g_1\|_{X^{s,b}_{\pm_1}} \|g_2\|_{X^{s,b}_{\pm_2}} \\
& \lesssim N_1^{-s-\half+} (L_0 L_1 L_2)^{0-} \|f\|_{X^{-s,1-b-}_{\pm,c}} \|g_1\|_{X^{s,b}_{\pm_1}} \|g_2\|_{X^{s,b}_{\pm_2}}  \, ,
\end{align*}
if $\epsilon = b-\half+$ . This implies the desired bound, if $s > -\half$ .\\	
Case 3.3: $N_0 \lesssim L_2$ .
We obtain
\begin{align*}
&N_0 N_1^{-1} N_2^{-1} |\int f g_1 g_2 \, dx dt| \le N_0 N_1^{-1} N_2^{-1} \|g_2\|_{L^2} \|P_{K_{N_2,L_2}} (f g_1)\|_{L^2} \\
& \lesssim N_0 N_1^{-1} N_2^{-1} N_2^{\half + 2\epsilon} N_1^{\half-\epsilon} L_0^{\half-\epsilon} L_1^{\half}N_0^s L_0^{b-1+}  N_1^{-s}L_1^{-b}  N_2^{-s} N_0^{-b+}  L_2^{0-} \\
& \hspace{10em} \|f\|_{X^{-s,1-b-}_{\pm,c}} \|g_1\|_{X^{s,b}_{\pm_1}} \|g_2\|_{X^{s,b}_{\pm_2}} \\
& \lesssim N_2^{-\half-s+2\epsilon} N_1^{\half-b- \epsilon+} (L_0 L_1 L_2)^{0-} \|f\|_{X^{-s,1-b-}_{\pm,c}} \|g_1\|_{X^{s,b}_{\pm_1}} \|g_2\|_{X^{s,b}_{\pm_2}} \\
& \lesssim N_1^{-s-\half+} (L_0 L_1 L_2)^{0-} \|f\|_{X^{-s,1-b-}_{\pm,c}} \|g_1\|_{X^{s,b}_{\pm_1}} \|g_2\|_{X^{s,b}_{\pm_2}}  \, ,
\end{align*}
if $\epsilon = b-\half+$ . This implies the desired bound, if $s > -\half$ .\\
Dyadic summation gives the claimed result.
\end{proof}

{\bf Remark:} These results imply, that the estimates (\ref{1.5}) and (\ref{1.6}) are true , provided $ s >-\half$ , and either $\pm_1 = \pm_2$ or $\pm_1 \neq \pm_2$ and $|\xi_1| \nsim |\xi_2|$ in the sense of Lemma \ref{Lemma1.2}, namely $|\xi_2| \le \frac{1-c}{2(1+c)} |\xi_1|$ or  $|\xi_1| \le \frac{1-c}{2(1+c)}|\xi_2|$ .

\section{The case $\pm_1 \neq \pm_2$}

It suffices to prove
\begin{equation}
\label{2.0}
\sum_{1 \le N_0,N_1,N_2,L_0,L_1,L_2} I_1 \lesssim \|f\|_{X^{s,b}_{\pm,c}} \|g_1\|_{X^{s,b}_-}
\|g_2\|_{X^{-s,1-b-}_+} \, , 
\end{equation}
where 
$$ I_1 := N_1^{-1} \int (P_{K^{\pm,c}_{N_0,L_0}} f) (P_{K_{N_1,L_1}^-}g_1)(P_{K_{N_2,L_2}^+} g_2) dt dx \, , $$
and
\begin{equation}
\label{2.0'}
\sum_{1 \le N_0,N_1,N_2,L_0,L_1,L_2} I_2 \lesssim \|f\|_{X^{-s,1-b-}_{\pm,c}} \|g_1\|_{X^{s,b}_-}
\|g_2\|_{X^{s,b}_+} \, , 
\end{equation}
where 
$$ I_2 := N_0 N_1^{-1}N_2^{-1} \int (P_{K^{\pm,c}_{N_0,L_0}} f) (P_{K_{N_1,L_1}^-}g_1)(P_{K_{N_2,L_2}^+} g_2) dt dx \, . $$

It remains to consider the case $1 \le N_0 \lesssim N_1 \sim N_2$ , more precisely we may assume $|\xi_2| \ge \frac{1-c}{2(1+c)} |\xi_1|$ and  $|\xi_1| \ge \frac{1-c}{2(1+c)}|\xi_2|$. These assumptions imply  
\begin{equation}
\label{N1N2}
N_1 \le  2^4(1-c)^{-1} N_2 \quad{\mbox and}\quad N_2 \le  2^4(1-c)^{-1} N_1 \, ,
\end{equation}
as one easily checks.\\[0.2em]

In this case we need a further decomposition with respect to angular variables.
Similar decompositions were also used by Bejenaru-Herr \cite{BH} and Kinoshita \cite{K}.

 Decompose $\mathcal{S}^2$ by $\{\omega^j_A\}_{j \in \Omega_A}$ for each $A \in {\mathbb N}$ with the properties
\begin{enumerate}
	\item $\angle (x,y) \le A^{-1}$ $\forall x,y \in \omega^j_A$ ,
	\item ${\mathcal S}^2 = \cup_{j \in \Omega_A} \omega^j_A$ almost disjoint, i.e.  $1 \le \sum_{j \in \Omega_A} \chi_{\omega^j_A} (x) \le 3$ $\forall x \in {\mathcal S} \, . $ Any two centers of $\omega^j_A$ are separated by a distance $ \sim A^{-1}$.
\end{enumerate}

Define 
$$ \alpha(j_1,j_2) = \inf \{ |\angle (x,y)| : x \in \omega^{j_1}_A , y \in \omega^{j_2}_A\} $$ 
and 
$$ Q^j_A = \{(\xi,\tau) \in (\mathbb R^3  \setminus\{0\}) \times \mathbb R : \frac{\xi}{|\xi|} \in \omega^j_A \} \, . $$\\[0.2em]
{\bf First we consider the case $ 0 \le \angle(\xi_1,\xi_2) \le \frac{\pi}{2}$ .}

\begin{prop}
	\label{Prop.2.1}
Let $f,g_1,g_2 \in L^2$ and
$$ supp\, f \in K^{\pm,c}_{N_0,L_0} \, , \, supp \, g_1 \in Q^{j_1}_A \cap K^-_{N_1,L_1} \, , \, supp \, g_2 \in Q^{j_2}_A \cap K^+_{N_2,L_2} $$
and $ 1 \ll N_0 \lesssim N_1 \sim N_2$ , more precisely $
N_1 \le  2^4(1-c)^{-1} N_2 $ and $ N_2 \le  2^4(1-c)^{-1} N_1 $ ,
 moreover assume $ 8 \le A $
  , $\frac{1}{2A} \le \alpha(j_1,j_2) \le \frac{2}{A} $ . Then
$$ I:= \int f(\xi_1+\xi_2,\tau_1+\tau_2) g_1(\xi_1,\tau_1) g_2(\xi_2,\tau_2) d\xi_1 d\tau_1 d\xi_2 d\tau_2  $$
satisfies
$$ |I(f,g_1,g_2)| \lesssim N_1^{\half} (L_0 L_1 L_2)^{\half} \|f\|_{L^2} \|g_1\|_{L^2} \|g_2\|_{L^2} \, . $$
\end{prop}

\begin{proof}
We apply the transformation $\tau_1=|\xi_1|+c_1$ , $\tau_2 = |\xi_2| + c_2$ for fixed $\xi_1,\xi_2$ . By decomposing $f$ into $L_0$ pieces we reduce to proving
\begin{align}
\nonumber
\left|\int f(\phi^+_{c_1}(\xi_1) + \phi^+_{c_2}(\xi_2)) g_1(\phi^+_{c_1}(\xi_1)) g_2(\phi^-_{c_2}(\xi_2)) \, d\xi_1 d\xi_2 \right| \\\label{2.1}
\lesssim N_1^{\half} \|f\|_{L^2_{\xi \tau}} \|g_1 \circ \phi^+_{c_1}\|_{L^2_{\xi \tau}} \|g_2 \circ \phi^-_{c_2}\|_{L^2_{\xi \tau}} \, ,
\end{align}
where $\phi^{\pm}_{c_k}(\xi) := (\xi, \pm |\xi|+c_k)$ and
$$ supp \, f \subset \{(\xi,\tau) \in Q^j_A \, , \, c_0 \le \tau \pm c|\xi|\le c_0+1 \} \, , $$
$c_0$ fixed, with an implicit constant independent of $c_0$ . We may assume $(\xi,\tau) \in Q^j_A$ here with $\alpha(j_1,j) \sim A^{-1}$ by further decomposition, because $(\xi_k,\tau_k) \in Q^{j_k}_A$ with $\alpha(j_1,j_2) \sim A^{-1}$ . 

We scale $(\xi,\tau) \mapsto (N_1 \xi,N_1 \tau)$ and define
$$\tilde{f}(\xi,\tau) = f(N_1 \xi, N_1 \tau) \, , \, \tilde{g_k}(\xi_k,\tau_k) = g_k(N_1 \xi_k, N_1 \tau_k) \, , \, \tilde{c_k} = \frac{c_k}{N_1} \, . $$
$(\ref{2.1})$ reduces to
\begin{align}
\nonumber
\left|\int \tilde{f}(\phi^+_{\tilde{c_1}}(\xi_1) +\phi^+_{\tilde{c_2}}(\xi_2)) \tilde{g_1}(\phi^+_{\tilde{c_1}}(\xi_1)) \tilde{g_2}(\phi^-_{\tilde{c_2}}(\xi_2)) \, d\xi_1 d\xi_2 \right| \\
\label{2.2}
	\lesssim N_1^{-\half} \|\tilde{f}\|_{L^2_{\xi \tau}} \|\tilde{g_1} \circ \phi^+_{\tilde{c_1}}\|_{L^2_{\xi \tau}} \|\tilde{g_2} \circ \phi^-_{\tilde{c_2}}\|_{L^2_{\xi \tau}} \, .
\end{align}
Namely, if $(\ref{2.2})$ is satisfied, we obtain by defining $\tilde{\xi_k} = N_1 \xi_k$ that the left hand side of (\ref{2.2}) equals
$$ \left|\int f(\tilde{\xi_1}+\tilde{\xi_2},|\tilde{\xi_1}|-|\tilde{\xi_2}| + \tilde{c_1}-\tilde{c_2}) \tilde{g_1}(\tilde{\xi_1},|\tilde{\xi_1}|+\tilde{c_1}) \tilde{g_2}(\tilde{\xi_2},|\tilde{\xi_2}|+\tilde{c_2}) \, d\tilde{\xi_1} d\tilde{\xi_2} \right| N_1^{-6} \, . $$
Thus \begin{align*}
LHS \,of\, (\ref{2.1})
 & \lesssim N_1^6 N_1^{-\half} \|\tilde{f}\|_{L^2_{\xi \tau}} \|\tilde{g_1} \circ \phi^+_{\tilde{c_1}} \|_{L^2_{\xi}}  \|\tilde{g_2} \circ \phi^-_{\tilde{c_2}} \|_{L^2_{\xi}} \\
 & = N_1^6 N_1^{-\half} \|f(N_1\xi,N_1 \tau)\|_{L^2_{\xi \tau}} \|g_1(N_1\xi_1,N_1 |\xi_1|+c_1)\|_{L^2_{\xi_1}} \\
 & \quad \quad \quad \quad \quad \quad \cdot\|g_2(N_1\xi_2,-N_1 |\xi_2|+c_2)\|_{L^2_{\xi_2}} \\
 & \lesssim N_1^6 N_1^{-\half} N_1^{-2} \|f\|_{L^2_{\xi \tau}} N_1^{-\frac{3}{2}} \|g_1 \circ \phi^+_{c_1}\|_{L^2_{\xi}}N_1^{-\frac{3}{2}} \|g_2 \circ \phi^-_{c_2}\|_{L^2_{\xi}} \\
 & = N_1^{\half} \|f\|_{L^2_{\xi \tau}}  \|g_1 \circ \phi^+_{c_1}\|_{L^2_{\xi}} \|g_2 \circ \phi^-_{c_2}\|_{L^2_{\xi}} \, .
\end{align*}
Thus (\ref{2.1}) follows from the estimate
$$ \| \tilde{g_1}_{| S_1} \ast \tilde{g_2}_{| S_2}\|_{L^2(S_3^{\pm}(N_1^{-1}))} \lesssim N_1^{-\half} \|\tilde{g_1}\|_{L^2(S_1)} \|\tilde{g_2}\|_{L^2(S_2]} \, , $$
which we now prove. Here we use our assumption (\ref{N1N2}) and obtain
\begin{align*}
S_1 & = \{(\xi_1,\tau_1) \in Q^{j_1}_A \, , \, \half \le |\xi_1| \le 2 \, , \, \tau_1=|\xi_1|+\tilde{c_1} \} \\
S_2 & = \{(\xi_2,\tau_2) \in Q^{j_2}_A \, , \, (1-c)2^{-5} \le |\xi_2| \le 2^5 (1-c)^{-1} \, , \, \tau_2=-|\xi_2|+\tilde{c_2} \}
\end{align*}
Note that $\tilde{f}$ is supported in
$$ S^{\pm}_3(N_1^{-1}) = \{(\xi,\tau) \in Q^j_A \, , \, \frac{N_0}{2N_1} \le |\xi| \le \frac{2N_0}{N_1} \, , \, \psi^{\pm}(\xi) \le \tau \le \psi^{\pm}(\xi) + \frac{1}{N_1} \} \, ,$$
where 
$$ \psi^{\pm}(\xi) := \mp c|\xi|+\frac{c_0}{N_1} \, . $$
We further decompose $f,g_1,g_2$ as follows into a finite number of pieces:
$$ f = \sum_{j' = j^0}^{j^0 +k} \chi_{Q^{j'}_{kA}} f \, , \, g_1 =  \sum_{j_1' = j_1^0}^{j_1^0 +k} \chi_{Q^{j_1'}_{kA}} g_1 \, , \,  g_2 =  \sum_{j_2' = j_2^0}^{j_2^0 +k} \chi_{Q^{j_2'}_{kA}} g_2 \, , $$
where $ k = 2^{40} (1-c)^{-2} $ . Thus we may assume that
$$ supp \, f \subset Q^{j'}_{kA} \, , \, supp \,g_1 \subset Q^{j_1'}_{kA} \, , \, supp \,g_2 \subset Q^{j_2'}_{kA} $$
with fixed $j',j_1',j_2'$ , $\frac{1}{2A} \le \alpha(j',j_1') \le \frac{2}{A}$ , $\frac{1}{2A} \le \alpha(j',j_2') \le \frac{2}{A}$ , $\frac{1}{2A} \le \alpha(j_1',j_2') \le \frac{2}{A}$. 

This implies that we may assume in $S^{\pm}_3(N_1^{-1})$ that $\half \le |\xi|$ , because $\xi=\xi_1+\xi_2$ with $\xi_l \in Q^{j_l'}_{kA}$ . Namely, if $\xi_1=(|\xi_1|,0,0)$ by rotation and $\angle(\xi_1,\xi_2) \le A^{-1}$, then
$$ \frac{||\xi_2|-\xi_{21}|}{|\xi_2|} \le |(1,0,0)-\frac{\xi_2}{|\xi_2|} | \le A^{-1} \, ,$$  which (for $A\ge 2$) implies $\xi_{21} \ge 0$ , thus 
$$|\xi|=|\xi_1+\xi_2| = |(|\xi_1|+\xi_{21},\xi_{22},\xi_{23})| \ge |\xi_1| \ge \half \, .$$ 
This means that we may consider from now on
\begin{align*} 
S_1 & = \{(\xi_1,\tau_1) \in Q^{j_1'}_{kA} \, , \, \half \le |\xi_1| \le 2 \, , \, \tau_1=|\xi_1|+\tilde{c_1} \} \\
S_2 & = \{(\xi_2,\tau_2) \in Q^{j_2'}_{kA} \, , \, (1-c)2^{-4} \le |\xi_2| \le 2^4 (1-c)^{-1}\, , \, \tau_2=-|\xi_2|+\tilde{c_2} \} \\
 S^{\pm}_3(N_1^{-1}) &= \{(\xi,\tau) \in Q^{j'}_{kA} \, , \, \half \le |\xi| \le 2^5(1-c)^{-1} \, , \, \psi^{\pm}(\xi) \le \tau \le \psi^{\pm}(\xi) + \frac{1}{N_1} \} \, .
\end{align*}
Defining
$$ S_3^{\pm,h} = \{(\xi,\tau) \in Q^{j'}_{kA} \, , \, \half \le |\xi| \le 2^5(1-c)^{-1} \, , \, \psi_{\pm}(\xi)-\tau = h \} \, ,$$  where $0 \le h \le N_1^{-1} \, , $
we obtain
$$\|\phi\|_{L^2(S^{\pm}_3(N_1^{-1}))} = \left(\int_0^{N_1^{-1}} \|\phi\|^2_{L^2(S_3^{\pm,h})} dh \right)^{\half} \le N_1^{-\half} \sup_h \|\phi\|_{L^2(S_3^{\pm,h})} \, . $$
Therefore it suffices to prove
\begin{equation}
\label{2.3}\| \tilde{g_1}_{|S_1} \ast \tilde{g_2}_{|S_2}\|_{L^2(S_3^{\pm,h})} \lesssim \|\tilde{g_1}\|_{L^2(S_1)} \|\tilde{g_2}\|_{L^2(S_2)} \, .
\end{equation}
The unit normals on $S_1,S_2,S_3^{\pm,h}$ are given by
$$n_{S_1}(\xi_1) = 2^{-\half} (-\frac{\xi_1}{|\xi_1|},1) \, , \, n_{S_2}(\xi_2) = 2^{-\half} (\frac{\xi_2}{|\xi_2|},1) \, , \, n_{S_3}(\xi) = (1+c^2)^{-\half} ( \pm c \frac{\xi}{|\xi|},1) \, . $$
We now define another normal on $S_3^{\pm,h}$ orthogonal to $n_{S_3}(\xi)$ . Let $v$ be a unit vector. Then
$$ |\det(\frac{\xi_1}{|\xi_1|},\frac{\xi_2}{|\xi_2|},v)| = |\langle v, \frac{\xi_1}{|\xi_1|} \times \frac{\xi_2}{|\xi_2|} \rangle |\, . $$
Let $ \frac{{\xi_i}^0}{|{\xi_i}^0|}$  be the centers of $\omega^{j_i'}_{kA}$ .
By the assumption $\alpha(j_1',j_2') \ge \frac{1}{2A}$ we obtain
\begin{equation}
\label{2.3'}
\left| \frac{{\xi_1}^0}{|{\xi_1}^0|} \times \frac{{\xi_2}^0}{|{\xi_2}^0|} \right| = \left| \sin \angle(\frac{{\xi_1}^0}{|{\xi_1}^0|},\frac{{\xi_2}^0}{{|\xi_2}^0|})\right| \ge \frac{1}{3A} \, .
	\end{equation}
Define
$$ v = \frac{ \frac{{\xi_1}^0}{|{\xi_1}^0|} \times \frac{{\xi_2}^0}{|{\xi_2}^0|} }{| \frac{{\xi_1}^0}{|{\xi_1}^0|} \times \frac{{\xi_2}^0}{|{\xi_2}^0|}|} \, .$$
This implies
$$ \left|\det(\frac{{\xi_1}^0}{{|\xi_1}^0|},\frac{{\xi_2}^0}{|{\xi_2}^0|},v)\right| = \left|\langle v, \frac{{\xi_1}^0}{|{\xi_1}^0|} \times \frac{{\xi_2}^0}{|{\xi_2}^0|} \rangle \right| \ge \frac{1}{3A} $$
and $\langle v, \frac{\pm {\xi_1}^0}{|{\xi_1}^0|} \rangle =0$ . Moreover we obtain $\forall \, \xi_1 \in Q^{j_1'}_{kA} \, , \, \xi_2 \in Q^{j_2'}_{kA}$ :
\begin{align*}
 &\left|\langle v, \frac{\xi_1}{|\xi_1|} \times \frac{\xi_2}{|\xi_2|} \rangle  - \langle v, \frac{{\xi_1}^0}{|{\xi_1}^0|} \times \frac{{\xi_2}^0}{|{\xi_2}^0|} \rangle\right| \\
 & \le \left|(\frac{\xi_1}{|\xi_1|}-\frac{{\xi_1}^0}{|{\xi_1}^0|})\times \frac{\xi_2}{|\xi_2|}\right| + \left|\frac{{\xi_1}^0}{|{\xi_1}^0|} \times (\frac{\xi_2}{|\xi_2|}-\frac{{\xi_2}^0}{|{\xi_2}^0|})\right| \le \frac{2}{kA} \, ,
\end{align*}
so that
\begin{equation}
\label{2.4}
 |\det(\frac{\xi_1}{|\xi_1|},\frac{\xi_2}{|\xi_2|},v)| \ge \frac{1}{4A} \, . 
\end{equation}

We foliate $S_3^{\pm,h}$ as follows: $S_3^{\pm,h} = \cup_{\tilde{d}} \, S_{3,\tilde{d}}^{\pm}$ , where
$$ S_{3,\tilde{d}}^{\pm} = S_3^{\pm,h} \cap \{\langle \xi,v \rangle = \tilde{d}\} \, . $$
For $\xi \in Q^{j'}_{kA}$ , $\xi_1 \in Q^{j_1'}_{kA}$ the estimate
$$\left|\langle v,\frac{\xi}{|\xi|}\rangle \right| = \left|\langle v,\frac{\xi}{|\xi|} \pm \frac{{\xi_1}^0}{|{\xi_1}^0|}\rangle\right| \le \min_{\pm}\left|\frac{\xi}{|\xi|} \pm \frac{{\xi_1}^0}{|{\xi_1}^0|}\right| \lesssim A^{-1} $$
applies, because $\alpha(j',j_1') \sim A^{-1}$ , so that using $|\xi| \lesssim 1$ :
$$ |\tilde{d}| = |\langle v,\xi \rangle| \lesssim |\xi| A^{-1} \lesssim A^{-1} \, . $$
This implies
$$ \|\phi\|_{L^2(S_3^{\pm,h})} = (\int_{-A^{-1}}^{A^{-1}} \|\phi\|_{L^2(S^{\pm}_{3,\tilde{d}})} d \tilde{d})^{\half} \lesssim A^{-\half} \sup_{\tilde{d}} \|\psi\|_{L^2(S^{\pm}_{3,\tilde{d}})} \, . $$
This reduces (\ref{2.3}) to
\begin{equation}
\label{2.5}
\| \tilde{g_1}_{| S_1} \ast \tilde{g_2}_{| S_2} \|_{L^2(S_{3,\tilde{d}}^{\pm})} \lesssim A^{\half} \|\tilde{g_1}\|_{L^2(S_1)} \|\tilde{g_2}\|_{L^2(S_2)} \, . 
\end{equation}
Another normal to ${S_{3,\tilde{d}}^{\pm}}$ is given by $n_{S_{3,\tilde{d}}} =(v,0)$ independently of $\tilde{d}$ . We obtain an orthonormal system by replacing $(v,0)$ by
	$$ n_{S_3'} = \frac{n_{S_{3,\tilde{d}}} - \langle n_{S_3},n_{{S_{3,\tilde{d}}}} \rangle n_{S_3}}{|{n_{S_{3,\tilde{d}}} - \langle n_{S_3},n_{{S_{3,\tilde{d}}}} \rangle n_{S_3}|}} \, , $$
	where we remark that the denominator fulfills
	\begin{align}
	\nonumber
2 &\ge |n_{S_{3,\tilde{d}}} - \langle n_{S_3},n_{{S_{3,\tilde{d}}}} \rangle n_{S_3}| 
	=|(v,0)\mp \frac{c^2}{1+c^2} \langle \frac{\xi}{|\xi|},v\rangle (\frac{\xi}{|\xi|},1)| \\
	\label{a}
&	\ge  |v|-\frac{c^2}{1+c^2}|\langle\frac{\xi}{|\xi|},v\rangle| \ge 1-\frac{c^2}{1+c^2} \ge \half
	\end{align}
We obtain
\begin{align*}
d &:= |\det(n_{S_1},n_{S_2},n_{S_3},n_{{S_3}'})| = \frac{|\det(n_{S_1},n_{S_2},n_{S_3},n_{S_{3,\tilde{d}}})|}{|n_{S_{3,\tilde{d}}}
 - \langle n_{S_3},n_{S_{3,\tilde{d}}} \rangle n_{S_3} |} \\
& \ge \half |\det(n_{S_1},n_{S_2},n_{S_3},n_{S_{3,\tilde{d}}})| 
= \frac{1}{4 \sqrt{1+c^2}}\left| \begin{array}{rrrr}                                
-\frac{\xi_1}{|\xi_1|} & \frac{\xi_2}{|\xi_2|} & \pm c \frac{\xi}{|\xi|} & v \\                                               
1 & 1 & 1 & 0 \\                                                                                             
\end{array}
\right| \\
& \ge \frac{1}{4 \sqrt{1+c^2}} (|\det(\frac{\xi_1}{|\xi_1|},\frac{\xi_2}{|\xi_2|},v)| - |R|) \, ,
\end{align*}	
where by $\xi=\xi_1+\xi_2$ we obtain
\begin{align*}
|R| & = |-\det(\frac{\xi_2}{|\xi_2|},\pm c \frac{\xi}{|\xi|},v) + \det(-\frac{\xi_1}{|\xi_1|},\pm c \frac{\xi}{|\xi|},v)| \\
& = |\mp \frac{c}{|\xi_2|\, |\xi|} \det(\xi_2,\xi,v) \mp \frac{c}{|\xi_1|\,|\xi|} \det(\xi_1,\xi,v)| \\
& = \frac{c}{|\xi|} \left| \frac{1}{|\xi_2|} \det(\xi_2,\xi_1,v) + \frac{1}{|\xi_1|} \det(\xi_1,\xi_2,v)\right| = \frac{c}{|\xi|}\left| \frac{1}{|\xi_1|} - \frac{1}{|\xi_2|} \right| |\det(\xi_1,\xi_2,v)| \\
& \le \frac{c}{|\xi|} \frac{||\xi_1|-|\xi_2||}{|\xi_1|\,|\xi_2|} |\det(\xi_1,\xi_2,v)|
\le c |\det(\frac{\xi_1}{|\xi_1|}\frac{\xi_2}{|\xi_2|},v)| \, .
 \end{align*}
 This implies by (\ref{2.4}):
 \begin{equation}
 \label{2.6}
 d \ge \frac{1-c}{4 \sqrt{1+c^2}}|\det(\frac{\xi_1}{|\xi_1|}\frac{\xi_2}{|\xi_2|},v)| \ge \frac{1-c}{2^5 A} \, .
 \end{equation}
 Let
 $$ \sigma_1=(\xi_1,|\xi_1|+\tilde{c_1}) \in S_1 \, , \, \sigma_2=(\xi_2,-|\xi_2|+\tilde{c_2}) \in S_2 \, , \, 
 \sigma_3=(\xi,\mp c|\xi|+\frac{c_0}{N_1} -h) \in S_3^{\pm,h} 
 $$
 and
$$ \sigma_1'=(\xi_1',|\xi_1'|+\tilde{c_1}) \in S_1 \, , \, \sigma_2'=(\xi_2',-|\xi_2'|+\tilde{c_2}) \in S_2 \, , \, 
\sigma_3'=(\xi',\mp c|\xi'|+\frac{c_0}{N_1} -h) \in S_3^{\pm,h} 
$$ 
We obtain by $S_1 \subset Q^{j_1'}_{kA}$ , $S_2 \subset Q^{j_2'}_{kA}$ , $ S_3^{\pm,h} \subset Q^{j'}_{kA}$ the estimates
\begin{align}
\label{2.7}
|n_1(\sigma_1) - n_1(\sigma_1')| &= 2^{-\half} \left|\frac{\xi_1}{|\xi_1|} -\frac{\xi_1'}{|\xi_1'|}  \right| \le \frac{1}{kA} \\
\label{2.8}
|n_2(\sigma_2) - n_2(\sigma_2')| &= 2^{-\half} \left|\frac{\xi_2}{|\xi_2|} -\frac{\xi_2'}{|\xi_2'|}  \right| \le \frac{1}{kA} \\
\label{2.9}
|n_3(\sigma_3) - n_3(\sigma_3')| &= \frac{c}{\sqrt{1+c^2}} \left|\frac{\xi}{|\xi|} -\frac{\xi'}{|\xi'|}  \right| \le \frac{1}{kA} \, .
\end{align}
Moreover by $|\xi_1| \le 2$ we obtain
\begin{align}
\nonumber
|\langle \sigma_1-\sigma_1',n_1(\sigma_1') \rangle| & \nonumber\le |\langle(\xi_1-\xi_1',|\xi_1|-|\xi_1'|),(-\frac{\xi_1'}{|\xi_1'|},1) \rangle| \\  \nonumber
& = |(-\frac{\langle \xi_1,\xi_1'\rangle}{|\xi_1'|} + |\xi_1'|) + (|\xi_1|-|\xi_1'|)| \\
\nonumber
& = |\xi_1| |\cos \angle(\xi_1,\xi_1') -1| \\
\label{2.7'}
& \le 2 \angle(\xi_1,\xi_1')^2 \le \frac{2}{k^2 A^2}
\end{align}
and similarly  by $|\xi_2| \le 2^5 (1-c)^{-1}$
$$|\langle \sigma_2-\sigma_2',n_2(\sigma_2') \rangle| \le\frac{2^5}{(1-c)k^2 A^2}$$
and
$$|\langle \sigma_3-\sigma_3',n_3(\sigma_3') \rangle| \le\frac{2^6}{(1-c)k^2 A^2}$$
Next we consider the normal $n_3'(\sigma_3) := n_{S_3'}(\sigma_3)$ .
By definition we obtain
$$ n_3'(\sigma_3) = \frac{(v,0)\mp \frac{c^2}{1+c^2} \langle \frac{\xi}{|\xi|},v\rangle (\frac{\xi}{|\xi|},1)}{|(v,0)\mp \frac{c^2}{1+c_2} \langle \frac{\xi}{|\xi|},v\rangle (\frac{\xi}{|\xi|},1)|} =: \frac{a}{|a|} $$								
and similarly $n_3'(\sigma_3') = \frac{b}{|b|}$ with analogously defined $b$. This implies
\begin{equation}
\label{2.9'}|n_3'(\sigma_3')-n_3'(\sigma_3)|  \le \frac{|a-b|}{|a|} + |b| \left|\frac{1}{|a|}-\frac{1}{|b|}\right| \le \frac{2|a-b|}{|a|}
\end{equation}								
By (\ref{a}) we obtain $\half \le |a| \lesssim 1$ and similarly $\half \le |b| \lesssim 1$ . Moreover
\begin{align*}
|a-b|  &\le \frac{c^2}{1+c^2}\left(|\langle \frac{\xi}{|\xi|},v\rangle -\langle \frac{\xi'}{|\xi'|},v \rangle| 
\, |\frac{\xi}{|\xi|}| +
|\langle \frac{\xi'}{|\xi'|},v \rangle|
 |\frac{\xi}{|\xi|} - \frac{\xi'}{|\xi'|}|
  +|\langle \frac{\xi}{|\xi|} - \frac{\xi'}{|\xi'|},v \rangle|\right) \\
  & \le \frac{3}{kA}
\end{align*}
and thus
\begin{align}
\label{2.10}
|n_3'(\sigma_3')-n_3'(\sigma_3)|  \le \frac{12}{kA}
\end{align}  
Next we want to prove
\begin{equation}
\label{2.8*}|\langle \sigma_3-\sigma_3',n_3'(\sigma_3')\rangle| \le \frac{2^7}{(1-c) k^2A^2} \, .
\end{equation}
We recall that by (\ref{a}) we have $|n_{S_3^{\tilde{d}}} - \langle n_{S_3},n_{S_3}^{\tilde{d}} \rangle n_{S_3}| \ge \half$ so that by the definition of $n_3'$ we obtain
\begin{align*}
&|\langle \sigma_3-\sigma_3',n_3'(\sigma_3')\rangle| \\& \le 2 |\langle(\xi-\xi',\mp c(|\xi|-|\xi'|)),(v-\frac{c^2}{1+c^2} \langle \frac{\xi'}{|\xi'|},v \rangle \frac{\xi'}{|\xi'|},\mp \frac{c}{1+c^2}\langle \frac{\xi'}{|\xi'|},v\rangle) \rangle| \\
&\le 2 |  \langle \xi - \xi',v-\frac{c^2}{1+c^2} \langle\frac{\xi'}{|\xi'|},v \rangle \frac{\xi'}{|\xi'|}\rangle + \frac{c^2}{1+c^2} \langle \frac{\xi'}{|\xi'|},v \rangle (|\xi|-|\xi'|)| \\
&\le 2 | \langle \xi - \xi',v  \rangle -\frac{c^2}{1+c^2} \langle\frac{\xi'}{|\xi'|},v \rangle
(\frac{\langle \xi, \xi' \rangle}{|\xi'|} - |\xi'|)
 + \frac{c^2}{1+c^2} \langle \frac{\xi'}{|\xi'|},v \rangle (|\xi|-|\xi'|)| \\
&\le 2 | \langle \xi - \xi',v  \rangle -\frac{c^2}{1+c^2} \langle\frac{\xi'}{|\xi'|},v \rangle
(\frac{\langle \xi, \xi' \rangle}{|\xi'|} - |\xi|)| \, . 
\end{align*}
The first term vanishes, because $\langle \xi,v\rangle = \langle \xi',v\rangle = \tilde{d}$ .
The second term is bounded by 
$$|\frac{\langle \xi, \xi' \rangle}{|\xi'|} - |\xi|| = |\xi| \,|\cos \angle(\xi,\xi') -1| \le 2^6 (1-c)^{-1} \angle(\xi,\xi')^2 \le \frac{2^6}{(1-c)k^2A^2} \,  $$
where we used $|\xi| \le 2^6(1-c)^{-1}$ .
This implies (\ref{2.8*}).

Next we want to show that we may assume that
\begin{equation}
\label{2.8'}
 |\langle \sigma_i - \sigma_i',n_j(\sigma_j')\rangle| \le \frac{2^7}{(1-c)k^2 A^2} \end{equation}
 and
 \begin{equation}
 \label{2.8''}
|\langle \sigma_i-\sigma_i',n_3'(\sigma_3')\rangle| \le \frac{2^8}{(1-c) k^2 A^2}  
\end{equation}
for $1\le i,j \le 3$ .
We only prove the first estimate, because the last one is treated similarly. Let $\delta := \frac{2^6}{(1-c) k^2 A^2}$ and $T_0(n) := \{ x \in \mathbb{R}^4 , |\langle n,x \rangle| \le \delta \}$, $T_{k'}(n) := T_0(n) + k' \delta n$ for $k' \in \mathbb{Z}$ , where $n$ is a unit normal on $\mathcal{S}^3$. Obviously $\cup_{k' \in\mathbb{Z}} T_{k'}(n) = \mathbb{R}^4$ and
$T_m(n)-T_{k'}(n) \subset \cup_{|l-(m-k')| \le 5} T_l(n)$ . Let $\sigma_1,\sigma_1' \in S_1$ , so that $\sigma_1 \in T_{k_1}(n_1(\sigma_1'))$, thus $|\langle \sigma_1,n_1(\sigma_1')\rangle-k_1 \delta| \le \delta$ . Using (\ref{2.7'}) this implies $$|\langle \sigma_1',n_1(\sigma_1')\rangle-k_1 \delta| \le |\langle \sigma_1'-\sigma_1,n_1(\sigma_1')\rangle| + |\langle \sigma_1,n_1(\sigma_1')\rangle -k_1 \delta| \le 2\delta \, ,$$ 
so that $\sigma_1' \in T_{k_1}(n_1(\sigma_1')) \cup T_{k_1-1}(n_1(\sigma_1')) \cup T_{k_1+1}(n_1(\sigma_1'))$ , essentially $S_1 \subset T_{k_1}(n_1(\sigma_1'))$ . Assume that we can show our desired estimate (\ref{2.5}) provided that $S_2 \subset T_{l_1}(n_1(\sigma_1'))$ and $S_3^{\pm} \subset T_{m_1}(n_1(\sigma_1'))$ , which means that
$$|\langle \sigma_2-\sigma_2',n_1(\sigma_1')\rangle| \le |\langle \sigma_2-l_1\delta,n_1(\sigma_1')\rangle| + |\langle \sigma_2'-l_1\delta,n_1(\sigma_1')\rangle| \le 2 \delta $$
and similarly 
$$ |\langle\sigma_3-\sigma_3',n_1(\sigma_1')\rangle| \le 2\delta \, ,
$$
so that $\forall \, m_1,l_1 \in\mathbb{Z}$ :
\begin{align*}
\| \tilde{g_1}_{| S_1} \ast \tilde{g_2}_{| S_2 \cap T_{l_1}(n_1(\sigma_1'))}\|^2_{L^2(S_3^{\pm} \cap T_{m_1}(n_1(\sigma_1')))} \lesssim A \|\tilde{g_1}\|^2_{L^2(S_1)} \|\tilde{g_2}\|^2_{L^2(S_2 \cap T_{l_1}(n_1(\sigma_1'))} \, .
\end{align*}
This implies
\begin{align*}
&\|\tilde{g_1}_{S_1} \ast \tilde{g_2}_{| S_2} \|^2_{L^2(S_3^{\pm})} \\
&\lesssim \sum_{m_1 \in \mathbb{Z}} \|\tilde{g_1}_{| S_1 \cap T_{k_1}(n_1(\sigma_1'))} 
\ast \tilde{g_2}_{| S_2\cap \cup_{|l_1-(m_1-k_1)|\le 5}
 T_{l_1}(n_1(\sigma_1'))} \|^2_{L^2(S_3^{\pm} \cap T_{m_1}(n_1(\sigma_1')))} \\
& \lesssim A \|\tilde{g_1}\|^2_{L^2(S_1)} \sum_{m_1 \in \mathbb{Z}} \|\tilde{g_2}\|^2_{L^2(S_2 \cap T_{m_1-k_1}(n_1(\sigma_1')))} \\
& \lesssim A \|\tilde{g_1}\|^2_{L^2(S_1)} \|\tilde{g_2}\|^2_{L^2(S_2)} \, ,
 \end{align*}
which is our desired estimate (\ref{2.5}). This means that we may assume that the estimates
$|\langle \sigma_2-\sigma_2',n_1(\sigma_1')\rangle| \le 2\delta$ and $|\langle \sigma_3-\sigma_3',n_1(\sigma_1')\rangle| \le 2\delta$ apply.  
Similarly we may further assume that $|\langle \sigma_1-\sigma_1',n_2(\sigma_2)\rangle| \le 2 \delta$ and $|\langle \sigma_3-\sigma_3',n_2(\sigma_2)\rangle| \le 2 \delta$ . Finally we obtain (\ref{2.8'}) and (\ref{2.8''}).

Let the invertible linear transformation $T: \mathbb{R}^4 \to \mathbb{R}^4$ be given by
$$ T= 2^{-40}(1-c)^3 A^{-2} (N^t)^{-1} \, , $$
where $N=(n_1(\sigma_1'),n_2(\sigma_2'),n_3(\sigma_3'),n_3'(\sigma_3'))$ , and $\tilde{S_j} := T^{-1} S_j$ . We want to apply \cite{BH}, Theorem 1.3 for these manifolds $\tilde{S_j}$ . We have to prove that the assumptions are fulfilled (cf. Assumption 1.1 in \cite{BH}). We have to prove:
\begin{enumerate}
	\item [(I)] $diam(\tilde{S_j}) \le 1$ ,
	\item [(II)]$\half \le \det(\tilde{n_1}(\tilde{\sigma_1}),\tilde{n_2}(\tilde{\sigma_2}),\tilde{n_3}(\tilde{\sigma_3}),\tilde{n_3'}(\tilde{\sigma_3})) \le 1$ $\forall \, \tilde{\sigma_j} \in \tilde{S_j}$ ,
	\item [(III)]  \begin{align*}
	\sup_{\tilde{\sigma_j},\tilde{\sigma_j^0}\in \tilde{S_j}}
	\frac{|\tilde{n_j}(\tilde{\sigma_j}) - \tilde{n_j}(\tilde{\sigma_j^0})|}{|\tilde{\sigma_j} - \tilde{\sigma_j^0}|} &
	+ \sup_{\tilde{\sigma_3},\tilde{\sigma_3^0}\in \tilde{S_3}}
	\frac{|\tilde{n_3'}(\tilde{\sigma_3}) - \tilde{n_3'}(\tilde{\sigma_3^0})|}{|\tilde{\sigma_3} - \tilde{\sigma_3^0}|} \\
	&+ \sup_{\tilde{\sigma_j} \in \tilde{S_j}} |\tilde{n_j}(\tilde{\sigma_j})| + \sup_{\tilde{\sigma_3} \in \tilde{S_3}} |\tilde{n_3'}(\tilde{\sigma_j})|  \le 1
		\end{align*}
\end{enumerate}	
Proof of (I). For arbitrary $\sigma_j,\sigma_j' \in S_j$ we obtain by (\ref{2.8'}),(\ref{2.8''}) and recalling $ k = 2^{40} (1-c)^{-2} $ :
\begin{align*}
|T^{-1}(\sigma_j-\sigma_j')| & = 2^{40} (1-c)^{-3} A^2 |(\langle n_1(\sigma_1'),\sigma_j-\sigma_j'\rangle, ... , \langle n_3'(\sigma_3'),\sigma_j-\sigma_j' \rangle)| \\& \le 2^{40} (1-c)^{-3} A^2 2^9 (1-c)^{-1} k^{-2}A^{-2}  \le 2^{-31} < 1 \,.
\end{align*}
Proof of (II). The unit normals on $\tilde{S_j}$ are given by
\begin{equation}
\label{*}
\tilde{n_j}(\tilde{\sigma_j}) = \frac{T^t n_j(T \tilde{\sigma_j})}{|T^t n_j(T \tilde{\sigma_j})|} = \frac{N^{-1} n_j(T \tilde{\sigma_j})}{|N^{-1} n_j(T\tilde{\sigma_j})|} \quad , \quad \tilde{n_3'}(\tilde{\sigma_3}) = \frac{N^{-1} n_3'(T \tilde{\sigma_3})}{|N^{-1} n_3'(T \tilde{\sigma_3})|} \, .
\end{equation}
This implies
\begin{equation}
\label{2.11}\tilde{n_j}(T^{-1} \sigma_j') = N^{-1} n_j(\sigma_j') = e_j \quad , \quad \tilde{n_3'}(T^{-1} \sigma_3') = N^{-1} n_3'(\sigma_3') = e_4 \, ,
\end{equation} 
Using (\ref{2.6}) we obtain
\begin{equation}
\label{2.12}
\|N^{-1}\| = \|(N^t)^{-1}\| \le \frac{\|N^t\|^3}{|\det N^t|} \le 2^{11} (1-c)^{-1} A \, .
\end{equation}
By the definition of $T$ this implies
\begin{equation}
\label{2.12'}
\|T\| \le 2^{-40} (1-c)^3 A^{-2} (1-c)^{-1} 2^{11} A \le 2^{-29}(1-c)^2 A^{-1} \, .
\end{equation}
Consequently by (\ref{2.11}),(\ref{2.12}),(\ref{2.7}),(\ref{2.8}),(\ref{2.9}),(\ref{2.10}) we obtain
\begin{align}
\label{2.13}
|N^{-1} n_j(\sigma_j) - e_j| &= |N^{-1}(n_j(\sigma_j)-n_j(\sigma_j'))| \le \|N^{-1}\| |n_j(\sigma_j) - n_j(\sigma_j')| \\
& \le 2^{11} (1-c)^{-1} A (kA)^{-1} = 2^{-29} (1-c) \le 2^{-29} \, , \\
|N^{-1} n_3'(\sigma_3) - e_4| & \le 2^{11} (1-c)^{-1} A \,12\, (kA)^{-1} \le 2^{-25} \, ,
\end{align}
which implies
$$||N^{-1} n_j(\sigma_j)| - 1| \le 2^{-29} \quad , \quad ||N^{-1}n_3'(\sigma_3)|-1| \le 2^{-25} \, , $$
By (\ref{*}) and (\ref{2.13}) we obtain
\begin{align}
\nonumber
|\tilde{n_j}(\tilde{\sigma_j}) - e_j| &= \left|\frac{N^{-1} n_j(T\tilde{\sigma_j})}{|N^{-1} n_j(T\tilde{\sigma_j})|} - e_j \right| \\
\nonumber
&\le \left|\frac{N^{-1} n_j(T\tilde{\sigma_j})}{|N^{-1} n_j(T\tilde{\sigma_j})|} - \frac{e_j}{|N^{-1} n_j(T\tilde{\sigma_j}|}\right| + \left|(\frac{1}{|N^{-1} n_j(T\tilde{\sigma_j}|} -1)e_j\right| \\
\label{2.14}
&\le 2\frac{2^{-29}}{1-2^{-29}} \le 2^{-27}
\end{align}
and similarly
\begin{equation}
\label{2.15}
|\tilde{n_3}'(\tilde{\sigma_3})-e_4| \le 2^{-23} \, .
\end{equation}
(\ref{2.14}) and (\ref{2.15}) immediately imply (II). \\
Proof of (III). Let $\tilde{\sigma_i} , \tilde{\sigma_i}^0 \in \tilde{S_i}$ . We obtain by (\ref{2.13}):
\begin{align*}
&|\tilde{n_i}(\tilde{\sigma_i})-\tilde{n_i}(\tilde{\sigma_i}^0)| \\
& = \left|\frac{N^{-1}n_i(T\tilde{\sigma_i})}{|N^{-1}n_i(T\tilde{\sigma_i})|} - \frac{N^{-1}n_i(T\tilde{\sigma_i}^0)}{|N^{-1}n_i(T\tilde{\sigma_i}^0)|}\right| \\
&\le |N^{-1} n_i(T\tilde{\sigma_i}) - N^{-1} n_i(T\tilde{\sigma_i}^0)| (1 +  |\frac{1}{|N^{-1} n_i(T\tilde{\sigma_i})|} -1|) \\
& + |N^{-1} n_i(T \tilde{\sigma_i}^0) (\frac{1}{|N^{-1} n_i(T \tilde{\sigma_i}|)} -\frac{1}{|N^{-1} n_i(T \tilde{\sigma_i}^0|)})| \\
&\le  |N^{-1} n_i(T\tilde{\sigma_i}) - N^{-1} n_i(T\tilde{\sigma_i}^0)|(1+ \frac{|1-|N^{-1}n_i(T\tilde{\sigma_i})||}{|N^{-1}n_i(T\tilde{\sigma_i})|} + \frac{1}{|N^{-1} n_i(T\tilde{\sigma_i})|}) \\
& \le  |N^{-1} n_i(T\tilde{\sigma_i}) - N^{-1} n_i(T\tilde{\sigma_i}^0)| (1+ \frac{2^{-9}}{1-2^{-9}}+\frac{1}{1-2^{-9}}) \\
& \le 3 |N^{-1} n_i(T\tilde{\sigma_i}) - N^{-1} n_i(T\tilde{\sigma_i}^0)|\, .
\end{align*}
This implies by (\ref{2.12}),(\ref{2.12'}):
\begin{align}
\nonumber
\frac{|\tilde{n_i}(\tilde{\sigma_i})-\tilde{n_i}(\tilde{\sigma_i}^0)|}{|\tilde{\sigma_i}-\tilde{\sigma_i}^0|} & \le 3 \|N^{-1}\|\frac{|n_i(T\tilde{\sigma_i})| - n_i(T\tilde{\sigma_i}^0)|}{|\tilde{\sigma_i}-\tilde{\sigma_i}^0|}\\
\nonumber
& \le 3 \|N^{-1}\| \|T\| \frac{|n_i(T\tilde{\sigma_i})| - n_i(T\tilde{\sigma_i}^0)|}{|T\tilde{\sigma_i}-T\tilde{\sigma_i}^0|}\\
\label{2.16}
& \le 2^{-18} (1-c) \frac{|n_i(\sigma_i)-n_i(\sigma_i^0)|}{|\sigma_i-\sigma_i^0|} \, .
\end{align}
Now we obtain by $|\xi_i|\ge (1-c)2^{-4}$ for $i=1,2,3$ :
$$|n_i(\sigma_i)-n_i(\sigma_i^0)| \le \left| \frac{\xi_i}{|\xi_i|} - \frac{\xi_1^0}{|\xi_i^0|} \right| \le 2 \frac{|\xi_i-\xi_i^0|}{|\xi_i|} \le 2^5 (1-c)^{-1} |\xi_i-\xi_i^0| \, .$$
By (\ref{2.16}) this implies
$$ \frac{|\tilde{n_i}(\tilde{\sigma_i})-\tilde{n_i}(\tilde{\sigma_i}^0)|}{|\tilde{\sigma_i}-\tilde{\sigma_i}^0|} \le 2^{-13}  \le  1 \, . $$
It remains to estimate $|\tilde{n_3}'(\tilde{\sigma_3}) -\tilde{n_3}'(\tilde{\sigma_3}^0)|$ . By (\ref{2.9'}) we obtain (with the notations used there)
$$
|n_3'(\sigma_3) -n_3'(\sigma_3^0)|  \le \frac{2|a-b|}{|a|} \, . $$ 
The numerator is 
$$\frac{2c}{1+c^2}|\langle \frac{\xi_3}{|\xi_3|},v \rangle (c \frac{\xi_3}{|\xi_3|},1) -\langle \frac{\xi_3^0}{|\xi_3^0|},v \rangle (c \frac{\xi_3^0}{|\xi_3^0|},1)| \le 6 \left| \frac{\xi_3}{|\xi_3|} -\frac{\xi_3^0}{|\xi_3^0|} \right| \, .$$ Because the denominator fulfills $|a| \ge \half$ we obtain
$$|n_3'(\sigma_3) -n_3'(\sigma_3^0)| \le 12 \left| \frac{\xi_3}{|\xi_3|} -\frac{\xi_3^0}{|\xi_3^0|} \right| \le 24 \frac{|\xi_3-\xi_3^0|}{|\xi_3|} \le 48 |\xi_3-\xi_3^0| $$
using $|\xi_3| \ge \half$ . 
Moreover we have 
$$ |\sigma_3-\sigma_3^0| \ge |\xi_3-\xi_3^0|$$
and gain
$$ \frac{|n_3'(\sigma_3) -n_3'(\sigma_3^0)|}{|\sigma_3-\sigma_3^0|} \le 48$$
and finally similarly as in (\ref{2.16}) the desired estimate
$$\frac{|\tilde{n_3}'(\tilde{\sigma_3}) - \tilde{n_3}'(\tilde{\sigma_3}^0)|}{|\tilde{\sigma_3}-\tilde{\sigma_3}^0|} \le 1 \, . $$
Thus the manifolds $\tilde{S_1},\tilde{S_2},\tilde{S_3}$ fulfill the assumption 1.1 of \cite{BH} with parameters $R = 1$ , $b = 1$ , $\Theta = \half$ and $\beta=1$ , so that we obtain by \cite{BH}, Theorem 1.3:
$$ \| f_{| \tilde{S_1}} \ast g_{| \tilde{S_2}} \|_{L^2(\tilde{S_3})} \lesssim \|f\|_{L^2(\tilde{S_1})} \|g\|_{L^2(\tilde{S_2})} \, . $$
By application of the linear invertible mapping $T$ we obtain by \cite{BH}, Proposition 1.2 the estimate (\ref{2.5}), namely
$$ \|f_{|S_1} \ast g_{|S_2}\|_{L^2(S_{3,\tilde{d}}^{\pm})} \lesssim d^{-\half} \|f\|_{L^2(S_1)} \|g\|_{L^(S_2)} \lesssim A^{\half} \|f\|_{L^2(S_1)} \|g\|_{L^2(S_2)} \, , $$
where $d = |\det(n_{S_1},n_{S_2},n_{S_3},n_{S_3}')| \ge (1-c) 2^{-4} A^{-1}$  by (\ref{2.6}). As proven before this implies the claimed estimate and completes the proof.
\end{proof}

Next we have to consider the case of a relatively small angle $\angle(\xi_1,\xi_2)$ .
\begin{prop}
	\label{Prop.2.2}
	Let $f,g_1,g_2 \in L^2$ and $supp \, f \subset K^{\pm,c}_{N_0,L_0}$ , $supp \, g_1 \subset Q^{j_1}_A \cap K^-_{N_1,L_1}$, $supp \, g_2 \subset Q^{j_2}_A \cap K^+_{N_2,L_2}$  and $1 < N_0 \lesssim N_1 \sim N_2$ , $A \gtrsim N_1^{\frac{3}{2}}$ .
	 Then the following estimate applies:
	$$|I(f,g_1,g_2)| \lesssim \min(L_0,L_1,L_2)^{\half} \|f\|_{L^2} \|g_1\|_{L^2} \|g_2\|_{L^2}  \, , $$
	where
	\begin{equation}
	\label{I}
	 I:= \int f(\xi_1+\xi_2,\tau_1+\tau_2) g_1(\xi_1,\tau_1) g_2(\xi_2,\tau_2) d\xi_1 d\tau_1 d\xi_2 d\tau_2  \, .
	 \end{equation}
\end{prop}
\begin{proof} Assume $L_1 \le L_2$ . Then the claimed estimate reduces to
	$$\|P_{K^+_{N_2,L_2}}((P_{K^{\pm,c}_{N_0,L_0}} f)(P_{K^-_{N_1,L_1}} g))\|_{L^2_{xt}} \lesssim \min(L_0, L_1)^{\half} \|f\|_{L^2_{xt}} \|g\|_{L^2_{xt}} \, .$$
The left hand side equals
\begin{align*}
&\| \chi_{K^+_{N_2,L_2}} (( \chi_{K^{\pm,c}_{N_0,L_0}} \widehat{f})\ast(\chi_{K^-_{N_1,L_1}} \widehat{g}))\|_{L^2_{\xi \tau}} \\& = \| \chi_{K^+_{N_2,L_2}}(\tau,\xi) \int (\chi_{K^{\pm,c}_{N_0,L_0}} \widehat{f})(\tau-\tau_1,\xi-\xi_1) (\chi_{K^-_{N_1,L_1}} \widehat{g})(\tau_1,\xi_1) d\tau_1 d\xi_1 \|_{L^2_{\xi\tau}}  \\
&\le \|\chi_{K^+_{N_2,L_2}}\left(\int |\widehat{f}|^2(\tau-\tau_1,\xi-\xi_1) |\widehat{g}|^2(\tau_1,\xi_1) d\tau_1 d\xi_1 \right)^{\half} |E(\tau,\xi)|^{\half}\|_{L^2_{\xi \tau}}\\
& \le \sup_{(\xi,\tau) \in K^+_{N_2,L_2}} |E(\tau,\xi)|^{\half} \| |\widehat{f}|^2 \ast |\widehat{g}|^2\|^{\half}_{L^1_{\xi \tau}} \\
& \le \sup_{(\xi,\tau) \in K^+_{N_2,L_2}} |E(\tau,\xi)|^{\half} \|f\|_{L^2_{xt}} \|g\|_{L^2_{xt}} \, ,
\end{align*}
where 
$$ E(\tau,\xi) := \{ (\tau_1,\xi_1) \in Q^{j_1}_A : \langle \tau-\tau_1\pm c|\xi-\xi_1| \rangle \sim L_0 , \langle \tau_1-|\xi_1| \rangle \sim L_1\} \, . $$
We have to prove that
$$ \sup_{(\xi,\tau) \in K^+_{N_2,L_2}} |E(\tau,\xi)| \lesssim \min(L_0, L_1) \, . $$
It is immediately clear that for fixed $\xi_1$ we obtain
$$ | \{\tau_1 : (\tau_1,\xi_1) \in E(\tau,\xi) \}| \lesssim \min (L_0,L_1) \, . $$
From $(\xi_1,\tau_1) \in Q^{j_1}_A \cap K^-_{N_1,L_1} $ we obtain :
$$ |\{\xi_1 : (\tau_1,\xi_1) \in E(\tau,\xi) \}| \lesssim N_1^3 A^{-2} \lesssim 1  $$
by our assumption $A \gtrsim N_1^{\frac{3}{2}}$ ,
so that
$$ \sup_{(\xi,\tau) \in K^+_{N_2,L_2}} |E(\tau,\xi)| \lesssim \min(L_0,L_1)  \, . $$
\end{proof}
\begin{Cor}
	\label{Cor.2.1}
	Let $f,g_1,g_2 \in L^2$ and assume $0 \le \angle(\xi_1,\xi_2) \le \frac{\pi}{2}$ ,  $supp \, f \subset K^{\pm,c}_{N_0,L_0}$ , $supp \, g_1 \subset  K^-_{N_1,L_1}$ , $supp \, g_2 \subset  K^+_{N_2,L_2}$  and $1 \le N_0 \lesssim N_1 \sim N_2$ , more precisely $N_1 \le 2^4(1-c)^{-1} N_2$ and  $N_2 \le 2^4(1-c)^{-1} N_1$. Then the following estimate applies:
	$$|I(f,g_1,g_2)| \lesssim N_1^{\half+} (L_0,L_1 L_2)^{\half} \|f\|_{L^2} \|g_1\|_{L^2} \|g_2\|_{L^2} \, ,$$
	where $I$ is given by (\ref{I}).
\end{Cor}
\begin{proof}
The case $1 \le N_0 \lesssim 1$ easily follows by H\"older's inequality. In the case $1 \ll N_0$
choose $M= N_1^{\frac{3}{2}}$ such that Prop. \ref{Prop.2.1} applies for $A < M$ and Prop. \ref{Prop.2.2} for $A=M$ . Then
\begin{align*}
|I(f,g_1,g_2)|  \lesssim &\sum_{A=1}^{M-1} \sum_{\alpha(j_1,j_2) \sim A^{-1}} |I(f, P_{Q^{j_1}_A} g_1,P_{Q^{j_2}_A g_2)})| \\
&+ \sum_{\alpha(j_1,j_2) \lesssim M^{-1}} |I(f,P_{Q^{j_1}_M} g_1,P_{Q^{j_2}_M} g_2)| \, , 
\end{align*}
The first term is handled by Prop. \ref{Prop.2.1} and Cauchy-Schwarz:
\begin{align*}
\sum_{\alpha(j_1,j_2) \sim A^{-1}} &|I(f,P_{ Q^{j_1}_A} g_1,P_{Q^{j_2}_A} g_2)|  \\
& \lesssim N_1^{\half} (L_0 L_1 L_2)^{\half} \sum_{\alpha(j_1,j_2) \sim A^{-1}} \|f\|_{L^2} \|P_{Q^{j_1}_A} g_1\|_{L^2} \|P_{Q^{j_2}_A} g_2\|_{L^2} \\
& \lesssim N_1^{\half} (L_0 L_1 L_2)^{\half} \|f\|_{L^2} \|g_1\|_{L^2} \|g_2\|_{L^2}
\end{align*}
for fixed $A$ . Dyadic summation with respect to $A$ implies the claimed estimate with an additional factor $N_1^{0+}$ . The second term is handled by Prop. \ref{Prop.2.2}.
\end{proof}

\begin{prop}
	Assume $0 \le \angle{\xi_1,\xi_2)} \le \frac{\pi}{2}$ . Let $f,g_1,g_2 \in L^2$ and assume  $supp \, f \subset K^{\pm,c}_{N_0,L_0}$ , $supp \, g_1 \subset  K^-_{N_1,L_1}$ , $supp \, g_2 \subset K^+_{N_2,L_2}$ and $1 \le N_0 \lesssim N_1 \sim N_2$ , more precisely  $N_1 \le 2^4(1-c)^{-1} N_2$ and  $N_2 \le 2^4(1-c)^{-1} N_1$ , $\max(L_0,L_1,L_2) \lesssim N_1$, $b=\half+$ and $-\half < s \le 0$ . Then the estimates (\ref{2.0}) and (\ref{2.0'}) are fulfilled.
\end{prop}
\begin{proof}
	Using $L_{max} \lesssim N_1$ and $b=\half+$ it suffices to prove for  $supp \, f \subset K^{\pm,c}_{N_0,L_0}$ , $supp \, g_1 \subset  K^-_{N_1,L_1}$ , $supp \, g_2 \subset K^+_{N_2,L_2}$ the estimate
	$$ |I_1(f,{g_1},{g_2})| \lesssim N_0^s N_1^{s-} N_2^{-s} (L_0 L_1 L_2)^{\half} \|f\|_{L^2} \|g_1\|_{L^2} \|g_2\|_{L^2} \, . $$
	By Cor. \ref{Cor.2.1} we obtain 
	$$
	 I_1 
	  \lesssim N_1^{-1} N_1^{\half +} (L_0L_1L_2)^{\half} \|f\|_{L^2} \|g_1\|_{L^2}\|g_2\|_{L^2} \, .
	 $$
	 It remains to prove $ N_1^{-\half+} \lesssim N_0^s N_1^{s-} N_2^{-s} \sim N_0^s N_1^{0-}$ , which applies for $ s > - \half$ .
	 Similarly we have to prove
 	$$ |I_2(f,{g_1},{g_2})| \lesssim N_0^{-s} N_1^{s-} N_2^{s} (L_0 L_1 L_2)^{\half} \|f\|_{L^2} \|g_1\|_{L^2} \|g_2\|_{L^2} \, . $$
 We obtain by Cor. \ref{Cor.2.1} :
 $$
 I_2 
 \lesssim N_0 N_1^{-1} N_2^{-1} N_1^{\half+} (L_0L_1L_2)^{\half} \|f\|_{L^2} \|g_1\|_{L^2}\|g_2\|_{L^2} \, .
 $$
 It remains to prove $ N_0 N_1^{-1} N_2^{-1} N_1^{\half+} \lesssim N_0^{-s} N_1^{s-} N_2^s \, \Leftrightarrow N_0^{1+s} \lesssim N_1^{2s+\frac{3}{2} -}$ , which applies for $ s > - \half$ .
 \end{proof}

\noindent {\bf Next we consider the case $\frac{\pi}{2} \le \angle(\xi_1,\xi_2) \le \pi$ .}

 We start with the case $\pi - \angle(\xi_1,\xi_2) \sim \frac{\pi}{A}$ and $A \le 2^5 (1-c)^{-\half} \frac{(N_1 N_2)^{\half}}{N_0}$ .
\begin{prop}
	\label{Prop.2.5}
	Let Let $f,g_1,g_2 \in L^2$ and assume  $supp \, f \subset K^{\pm,c}_{N_0,L_0}$ , $supp \, g_1 \subset Q^{j_1}_A \cap K^-_{N_1,L_1}$ , $supp \, g_2 \subset Q^{j_2}_A \cap K^-_{N_2,L_2}$ , $N_2 \ge 2^{-4} (1-c) N_1$ and $N_1 \ge 2^{-4} (1-c) N_2$, thus $N_1 \sim N_2$ , and $1 \ll N_0 $ , $4 \le A \le 2^5 (1-c)^{-\half} \frac{(N_1 N_2)^{\half}}{N_0}$ , $\frac{\pi}{2A} \le \pi - \alpha(j_1,j_2) \le \frac{2\pi}{A}$ . Then the following estimate applies:
	\begin{align*}
	I:=& |\int f(\xi_1+\xi_2,\tau_1+\tau_2) g_1(\xi_1,\tau_1) g_2(\xi_2,\tau_2) d\xi_1 d\tau_1 d\xi_2 d\tau_2|\\ &\lesssim N_0^{\half} A^{\frac{7}{8}} (L_0 L_1 L_2)^{\half} \|f\|_{L^2_{\xi \tau}} \|g_1\|_{L^2_{\xi \tau}} \|g_2\|_{L^2_{\xi \tau}} \, .
	\end{align*}
\end{prop}
\noindent Remark 1: We obtain
$$ N_0 \le 2|\xi|\le2(|\xi_1|+|\xi_2|) \le 4(N_1+N_2) \le 4(1+\frac{2^4}{1-c}) N_1 \le 2^7 (1-c)^{-1} N_1  \, .$$ 
	
\begin{proof}
	Define $\theta^{\pm}_0 \in (0,\pi)$ such that $\cos \theta^{\pm}_0 = \pm c$, and $\xi=\xi_1+\xi_2$. We distinguish three cases. \\
Case (I): $A > 2^{20}(1-c)^{-2}$ , $|\angle(\xi,\xi_1)-\theta_0^+| > 2^{10} (1-c)^{-1} A^{-\frac{3}{4}}$ and  $|\angle(\xi,\xi_1)-\theta_0^-| > 2^{10} (1-c)^{-1} A^{-\frac{3}{4}}$ . \\
Case (II): $A > 2^{20}(1-c)^{-2}$ , $|\angle(\xi,\xi_1)-\theta_0^+| \le 2^{10} (1-c)^{-1} A^{-\frac{3}{4}}$ or  $|\angle(\xi,\xi_1)-\theta_0^-| \le 2^{10} (1-c)^{-1} A^{-\frac{3}{4}}$ . \\
Case (III): $A \le 2^{20}(1-c)^{-2}$ . \\
{\bf Case (II):} The proof is similar to that of Prop. \ref{Prop.2.1}, but we use different decompositions. By the transformation $\tau_1=|\xi_1|+c_1$ , $\tau_2=|\xi_2|+c_2$ for fixed $\xi_1,\xi_2$ we may reduce to
\begin{align*}
& |\int f(\phi^+_{c_1}(\xi_1)+\phi^-_{c_2}(\xi_2)) g_1(\phi^+_{c_1}(\xi_1))g_2(\phi^-_{c_2}(\xi_2)) d\xi_1 d\xi_2 \\
&\lesssim N_0^{\half} A^{\frac{7}{8}} (L_0 L_1 L_2)^{\half} \|f\|_{L^2_{\xi \tau}} \|g_1\circ \phi^+_{c_1}\|_{L^2_{\xi}} \|g_2 \circ \phi^-_{c_2}\|_{L^2_{\xi}} \, ,
\end{align*}
where $\phi^{\pm}_{c_k}(\xi) = (\xi,\pm |\xi| + c_k)$ and $supp \, f \subset\{(\xi,\tau) : c_0 \le \tau+c|\xi| \le c_0+1 \} $, where $c_0$ is fixed. In fact we even prove this estimate with $A^{\frac{7}{8}}$ replaced by $A^{\frac{3}{4}}$ . Next we decompose $f,g_1,g_2$ as follows. Let $k=2^{40} (1-c)^{-2}$ , $\delta=2^{-40} (1-c)^2 N_0 A^{-\half}$  and
$$ f = \sum_{j\in \Omega_{kA}} \chi_{Q^j_{kA}} \sum _{l= - \left[\frac{N_0}{2\delta}\right]}^{l= \left[\frac{N_0}{\delta} \right]+1} \chi_{S^{N_0+ l \delta}_{\delta}} f \, ,$$
where $S^{\xi^0}_{\delta} := \{ (\xi,\tau) : \xi^0 \le |\xi| \le \xi^0+\delta \}$ . Because we are in Case II we obtain $\# j \sim A^{\half}$ . Obviously $\# l \sim A^{\half}$ , so that $\#(j,l) \sim A$ .
We may also easily reduce to considering $ g_1 = \chi_{Q^{j_1}_{kA}} g_1 $ , $g_2 = \chi_{Q^{j_1}_{kA}} g_2 $ . Thus it suffices to prove for $supp \, f \subset Q^j_{kA} \cap S^{N_0+l\delta}_{\delta}$ with fixed $l$ and $j$ the estimate
\begin{align*}
|\int f(\phi^+_{c_1}(\xi_1)+\phi^-_{c_2}(\xi_2)) g_1(\phi^+_{c_1}(\xi_1))g_2(\phi^-_{c_2}(\xi_2)) d\xi_1 d\xi_2| \\
\lesssim N_0^{\half} A^{\frac{1}{4}}\|f\|_{L^2_{\xi \tau}} \|g_1\circ \phi^+_{c_1}\|_{L^2_{\xi}} \|g_2 \circ \phi^-_{c_2}\|_{L^2_{\xi}} \, .
\end{align*}
We use the scaling $(\xi,\tau) \mapsto (N_0 \xi,N_0 \tau)$ and define
$$\tilde{f}(\xi,\tau) = f(N_0 \xi,N_0 \tau) \, , \, \tilde{g_i}(\xi_i,\tau_i) = g_k(N_0 \xi_i, N_0 \tau_i) \, , \, \tilde{c_i} = \frac{\xi_i}{N_0} \, . $$
Exactly as in Prop. \ref{Prop.2.1} we reduce to
\begin{align*}
|\int \tilde{f}(\phi^+_{\tilde{c_1}}(\xi_1)+\phi^-_{\tilde{c_2}}(\xi_2)) 
\tilde{g_1}(\phi^+_{\tilde{c_1}}(\xi_1))
\tilde{g_2}(\phi^-_{\tilde{c_2}}(\xi_2) d\xi_1 d\xi_2 | \\
\lesssim N_0^{-\half} A^{\frac{1}{4}}\|\tilde{f}\|_{L^2_{\xi \tau}} \|\tilde{g_1}\circ \phi^+_{\tilde{c_1}}\|_{L^2_{\xi}} \|\tilde{g_2} \circ \phi^-_{\tilde{c_2}}\|_{L^2_{\xi}} \, . 
\end{align*}
Here $\tilde{f}$ is supported in
$$S^{\mp}_3(N_0^{-1}) = \{ (\xi,\tau) \in Q^j_{kA} \cap S^{1+l\tilde {\delta}}_{\tilde{\delta}} : \psi^{\mp}(\xi) \le \tau \le \psi^{\mp}(\xi) + N_0^{-1} \} \, , $$
where $\psi^{\mp}(\xi):= \mp c|\xi| + c_0 N_0^{-1}$ and $\tilde{\delta} := N_0^{-1} \delta = 2^{-40} (1-c)^2 A^{-\half}$ .

Before we proceed we show that we may assume that there exist $\xi_1^0$ and $\xi_2^0$ such that
$|\xi_1^0| \sim N_1 N_0^{-1}$ , $|\xi_2^0| \sim N_2 N_0^{-1}$ with the property that $supp \,  \tilde{g_1} \circ \phi^+_{\tilde{c_1}} \subset B_{\tilde{\delta}}(\xi_1^0)$ ,  $supp \,  \tilde{g_2} \circ \phi^-_{\tilde{c_2}} \subset B_{\tilde{\delta}}(\xi_2^0)$ for the space variables, where  $B_{\tilde{\delta}}(\xi_j^0) = \{ \xi_j \in \mathbb{R}^3 : |\xi_j-\xi_j^0| \le \tilde{\delta}\}$ .

Define $C_{\tilde{\delta}}(\xi') := \{ (\xi,\tau) \in \mathbb{R}^4 . |\xi-\xi'| \le \tilde{\delta} \}$ and choose two sets $\{C_{\tilde{\delta}}(\xi_m')\}_m$ and $\{C_{\tilde{\delta}}(\xi_n'')\}_n$ , such that
$$ supp \, \tilde{g_1} \circ \phi^+_{\tilde{c_1}} \subset \cup_m C_{\tilde{\delta}}(\xi_m') \, , \,  supp \, \tilde{g_2} \circ \phi^+_{\tilde{c_2}} \subset \cup_n C_{\tilde{\delta}}(\xi_n') \, , $$
where $|\xi_m' - \xi_{m'}'| \ge \tilde{\delta}$ and $|\xi_n'' - \xi_{n'}''| \ge \tilde{\delta}$ . Thus $\# m \sim (\frac{N_1}{N_0} A^{\half})^2$ and $\# n \sim (\frac{N_2}{N_0} A^{\half})^2$. Now fix $m$. We claim that there exists only a finite number, which does not depend on $N_0,N_1,N_2$ or $m$, of indices $n$ , such that for some $\xi_1 \in C_{\tilde{\delta}}(\xi_m')$ there exists  $\xi_2 \in C_{\tilde{\delta}}(\xi_n'')$ with the property that $\xi_1+\xi_2 \in S^{1+l\tilde{\delta}}_{\tilde{\delta}} \cap Q^j_{kA}$ . Assume that $\xi_1,\xi_2$ fulfill this condition. Take any other pair $\xi_1',\xi_2'$ with this property. We obtain
\begin{align*}
&|(\xi_1+\xi_2) -(\xi_1'+\xi_2')|^2 = |\xi_1+\xi_2|^2 + |\xi_1'+\xi_2'|^2 - 2 \langle \xi_1+\xi_2,\xi_1'+\xi_2' \rangle \\
&= ||\xi_1+\xi_2|-|\xi_1'+\xi_2'||^2 + 2|\xi_1+\xi_2||\xi_1'+\xi_2'|(1- \cos \angle(\xi_1+\xi_2,\xi_1'+\xi_2')) \\
&\lesssim (2\tilde{\delta})^2 + 8 \angle(\xi_1+\xi_2,\xi_1'+\xi_2')^2 \le (2\tilde{\delta})^2 + 8 k^{-2} A^{-2} \\
&= 2^{-78} (1-c)^4 A^{-1} + 2^{-77}(1-c)^8 A^{-2} \le 2^{-76} (1-c)^4 A^{-1} = 2^{4} \tilde{\delta}^2 \, , 
\end{align*}
because $|\xi_1+\xi_2|,|\xi_1'  + \xi_2'| \le 1+ (l+1) \tilde{\delta} \le 2$ and $||\xi_1+\xi_2|-|\xi_1'+\xi_2'|| \le 2 \tilde{\delta}$ , if $\xi_1+\xi_2 , \xi_1'+\xi_2' \in S^{1+l\tilde{\delta}}_{\tilde{\delta}}$ and $\tilde{\delta} = 2^{-40} (1-c)^2 A^{-\half}$ .
This implies
$$|\xi_2-\xi_2'| \le |(\xi_1+\xi_2)-(\xi_1'+\xi_2')| + |\xi_1-\xi_1'| \le 4 \tilde{\delta} + 2 \tilde{\delta} = 6 \tilde{\delta} \, . $$
Because $\xi_2 \in C_{\tilde{\delta}}(\xi_n'')$ we obtain $\xi_2' \in C_{7\tilde \delta}(\xi_n'')$ , but $C_{7\tilde{\delta}}(\xi_n'') \cap C_{\tilde{\delta}}(\xi_{n'}'') \neq \emptyset$ only for a finite number, which does not depend on $N_0,N_1,N_2$ or $m$ ,  of indices $n'$ , because $|\xi_n'' - \xi_{n'}''| \ge \tilde{\delta}$ . This proves our claim. The same argument applies for exchanged roles of $g_1$ and $g_2$ .

Proceeding with the proof of the proposition this means that we only have to prove
$$ \|\tilde{g_1}_{| S_1} \ast \tilde{g_2}_{|S_2} \|_{L^2(S^{\mp}_3(N_0^{-1}))} \lesssim N_0^{-\half} A^{\frac{1}{4}} \|\tilde{g_1}\|_{L^2(S_1)} \|\tilde{g_2}\|_{L^2(S_2)} \, , $$
where 
$$
 S_i=\{(\xi_i,\tau_i)= \phi^+_{\tilde{c_i}}(\xi_i) \in Q^{j_i}_{kA} : \xi_i \in B_{\tilde{\delta}}(\xi_i^0) \} \, .$$
Defining
$$S^{\pm,h}_3 = \{(\xi,\tau) \in Q^j_{kA} \cap S^{1+l\tilde{\delta}}_{\tilde{\delta}} , \psi_{\mp}(\xi) - \tau = h \}$$
for $0 \le h \le N_0^{-1}$ , we obtain
$$\|\phi\|_{L^2(S^{\pm}_3(N_0^{-1}))} = (\int_0^{N_0^{-1}} \|\phi\|^2_{L^2(S^{\pm,h}_3)} dh)^{\half} \le N_0	^{-\half} \sup_h \|\phi\|_{L^2(S^{\pm,h}_3)}\, . $$
Thus we reduce to proving
\begin{equation}
\label{2.17}
\|\tilde{g_1}_{| S_1} \ast \tilde{g_2}_{|S_2} \|_{L^2(S^{\mp,h}_3)} \lesssim  A^{\frac{1}{4}} \|\tilde{g_1}\|_{L^2(S_1)} \|\tilde{g_2}\|_{L^2(S_2)} \, .
\end{equation} 
We remark that
\begin{equation}
\label{2.18}
\half \le 1+l\tilde{\delta} \le |\xi| \le 1+(l+1)\tilde{\delta} \le 2  \, .
\end{equation}
We obtain the same normals on $S_1,S_2,S_3^{\pm,h}$ as in the proof of Prop. \ref{Prop.2.1}. In order to construct the $4^{th}$ normal we foliate $S^{\pm,h}_3$ as $S_3^{\pm,h} = \cup_{\tilde{d}} S^{\pm}_{3,\tilde{d}}$ , where 
$$ S^{\pm}_{3,\tilde{d}} := S_3^{\pm,h} \cap \{ \langle \xi,v \rangle = \tilde{d} \} \, , $$
where $v$ is defined as in the proof of Prop. \ref{Prop.2.1}. Because $1+l\tilde{\delta} \le |\xi| \le 1+(l+1)\tilde{\delta}$ we know that $\tilde{d}$ varies in an interval $J$ of length
$|J| \le \tilde{\delta} \sim A^{-\half}$ . This implies
$$\|\phi\|_{L^2(S^{\pm,h}_3)} = (\int_J \|\phi\|^2_{L^2(S^{\pm}_{3,\tilde{d}})} d \tilde{d})^{\half} \lesssim A^{-\frac{1}{4}} \sup_{\tilde{d}} \|\phi\|_{L^2(S^{\pm}_{3,\tilde{d}})} \, . $$
This reduces (\ref{2.17}) to
\begin{equation}
\label{2.19}
\|\tilde{g_1}_{| S_1} \ast \tilde{g_2}_{|S_2} \|_{L^2(S^{\mp}_{3,\tilde{d}})} \lesssim  A^{\half} \|\tilde{g_1}\|_{L^2(S_1)} \|\tilde{g_2}\|_{L^2(S_2)} \, .
\end{equation}
We define the $4^{th}$ normal $n_3' := n_{S_3'}$ as in Prop. \ref{Prop.2.1}. Let $\sigma_j,\sigma_j' \in S_j$ $(j=1,2,3)$ , where $S_3:= S_{3,\tilde{d}}^{\pm}$ . Then
\begin{align*}
 \sigma_1 = & (\xi_1,|\xi_1|+\tilde{c_1}) \, , \, \sigma_1'=(\xi_1',|\xi_1'|+\tilde{c_1}) \\
 \sigma_2 = & (\xi_2,-|\xi_2|+\tilde{c_2}) \, , \, \sigma_2'=(\xi_2,-|\xi_2'|+\tilde{c_2}) \\
\sigma_3 = & (\xi,\mp c|\xi|+\frac{c_0}{N_0} -h) \, , \, \sigma_3'=(\xi',\mp c|\xi'|+\frac{c_0}{N_0}-h)  
\end{align*}
Because $\xi_1,\xi_1' \in B_{\tilde{\delta}}(\xi_1^0)$ , thus $|\xi_1| \sim |\xi_1^0| \sim N_1 N_0^{-1}$ , so that
\begin{align*}
|n_1(\sigma_1)-n_1(\sigma_1')| &= 2^{-\half} \left| \frac{\xi_1}{|\xi_1|} - \frac{\xi_1'}{|\xi_1'|}\right| = 2^{-\half} \frac{|\xi_1 |\xi_1'| - \xi_1' |\xi_1||}{|\xi_1||\xi_1'|} 
 \le 2^{\half} \frac{|\xi_1-\xi_1'|}{|\xi_1|} \\
 &\le 2^{\frac{3}{2}}\frac{\tilde{\delta}}{|\xi_1|} \le 8 \frac{\tilde{\delta} N_0}{N_1} \\
 & \le 2^{-37} (1-c)^2 A^{-\half} \frac{N_0}{N_1} \le 2^{-37}(1-c)^2A^{-\half} (\frac{2^4}{1-c})^{\half} \frac{N_0}{N_1^{\half}N_2^{\half}}\\
 & \le 2^{-37}(1-c)^2 A^{-\half}(\frac{2^4}{1-c})^{\half}2^5 (1-c)^{-\half} A^{-1} \\
 &\le 2^{-30} (1-c)  A^{-\frac{3}{2}} \, ,
\end{align*}
where we used $N_2 \ge 2^{-4} (1-c) N_1$ (cf. Remark) and  $A \le 2^5 \frac{(N_1)^{\half} (N_2)^{\half}}{N_0} (1-c)^{-\half}$ . Similarly we obtain
$$|n_2(\sigma_2)-n_2(\sigma_2')| \le 2^{-30} (1-c) A^{-\frac{3}{2}} \, . $$
By $S_3 \subset Q^j_{kA}$ we obtain
$$ |n_3(\sigma_3)-n_3(\sigma_3')| \le (kA)^{-1} = 2^{-40} (1-c)^2 A^{-1} \, . $$
As in the proof of Prop. \ref{Prop.2.1} we also obtain (cf. (\ref{2.10})):
$$|n_3'(\sigma_3)-n_3'(\sigma_3')| \le \frac{12}{kA} \le 2^{-36}(1-c)^2 A^{-1} \, . $$
By (\ref{2.7}) we obtain
\begin{align*}
|\langle \sigma_1-\sigma_1',n_1(\sigma_1') \rangle| & \le |\xi_1| \angle (\xi_1,\xi_1')^2 = |\xi_1| \left| \frac{\xi_1}{|\xi_1|} - \frac{\xi_1'}{|\xi_1'|}\right|^2 \\
& \le 2 |\xi_1|(\frac{2|\xi_1-\xi_1'|}{|\xi_1|})^2 \le 8 |\xi_1| (\frac{2 \tilde{\delta}}{|\xi_1|})^2 \le 2^6 \tilde{\delta}^2 \frac{N_0}{N_1} \\
& = 2^6 2^{-80} (1-c)^4 A^{-1} \frac{N_0}{N_1} 
\le 2^{-67} (1-c)^3 A^{-2} \, ,
\end{align*}
where we used $A \le 2^5 (1-c)^{-\half} \frac{(N_1 N_2)^{\half}}{N_0} \le 2^5 (1-c)^{-\half} \frac{2^2}{(1-c)^{\half}} \frac{N_1}{N_0} = 2^7 (1-c)^{-1} \frac{N_1}{N_0} $. Similarly we obtain
$$ |\langle \sigma_2-\sigma_2',n_2(\sigma_2') \rangle|  \le 2^{-67} (1-c)^3 A^{-2} \, . $$
Next we obtain as in the proof of Prop. \ref{Prop.2.1} :
$$
|\langle \sigma_3-\sigma_3',n_3(\sigma_3')\rangle| \le 2^6 (1-c)^{-1} (kA)^{-2} = 2^{-74} (1-c)^3 A^{-2}
$$
and also (cf. (\ref{2.8*})):
$$ |\langle \sigma_3-\sigma_3',n_3'(\sigma_3')\rangle| \le \frac{2^6}{(1-c) k^2 A^2} = 2^{-74}(1-c)^3 A^{-2}\, . $$
By the same arguments as in the proof of Prop. \ref{Prop.2.1} this implies that we may assume that the following estimates are fulfilled :
\begin{align}
\label{2.20}
 |\langle \sigma_i-\sigma_i',n_j(\sigma_j') \rangle|  &\le 2^{-66} (1-c)^3 A^{-2} \, , \\
 \label{2.21}
 |\langle \sigma_i-\sigma_i',n_3'(\sigma_3')\rangle| &\le  2^{-73} (1-c)^3 A^{-2}
\end{align}
for $1 \le i,j \le 3$ . 

Now we follow the proof of Prop. \ref{Prop.2.1}. Let the invertible linear transformation $T: \mathbb{R}^4 \to \mathbb{R}^4$ be given by
$$ T= 2^{-40}(1-c)^3 A^{-2} (N^t)^{-1} \, , $$
where $N=(n_1(\sigma_1'),n_2(\sigma_2'),n_3(\sigma_3'),n_3'(\sigma_3'))$ , and $\tilde{S_j} := T^{-1} S_j$ . We want to apply \cite{BH}, Theorem 1.3 for these manifolds $\tilde{S_j}$ . We have to prove that the assumptions are fulfilled (cf. Assumption 1.1 in \cite{BH}). Using the notation of Prop. \ref{Prop.2.1} we have to prove (I),(II),(III). \\
(I) : $diam(\tilde{S_j}) \le 1$ .\\ This is proven by (\ref{2.20}) and (\ref{2.21}) :
\begin{align*}
|T^{-1}(\sigma_j-\sigma_j')| & = 2^{40} (1-c)^{-3} A^2 |(\langle n_1(\sigma_1'),\sigma_j-\sigma_j'\rangle, ... , \langle n_3'(\sigma_3'),\sigma_j-\sigma_j' \rangle)| \\& \le 2^{40} (1-c)^{-3} A^2  2^{-64} A^{-2} (1-c)^3 \le 2^{-24} \,.
\end{align*}
(II): $\half \le \det(\tilde{n_1}(\tilde{\sigma_1}),\tilde{n_2}(\tilde{\sigma_2}),\tilde{n_3}(\tilde{\sigma_3}),\tilde{n_3'}(\tilde{\sigma_3})) \le 1$ $\forall \, \tilde{\sigma_j} \in \tilde{S_j}$ . \\
We remark that the lower bound (\ref{2.6}) for the determinant remains true, because (\ref{2.3'}) holds also in the case where  $\frac{\pi}{2A} \le \pi - \alpha(j_1,j_2) \le \frac{2\pi}{A}$ . Therefore
we obtain
$$ \|N^{-1}\| \le \frac{\|N^t\|^3}{|\det N^t|} \le 2^{11} (1-c)^{-1} A \quad , \quad \|T\| \le 2^{-29} (1-c)^2 A^{-1} \, . $$
Consequently
\begin{align*}
|N^{-1} n_j(\sigma_j) - e_j| &= |N^{-1}(n_j(\sigma_j)-n_j(\sigma_j'))| \le \|N^{-1}\| |n_j(\sigma_j) - n_j(\sigma_j')| \\
& \le 2^{11} (1-c)^{-1}\, A \, 2^{-30} (1-c) A^{-1} \le 2^{-19}  \, , \\
|N^{-1} n_3'(\sigma_3) - e_4| & \le 2^{11} (1-c)^{-1} A \, 2^{-36} (1-c)^2 A^{-1} \le 2^{-25} (1-c) \, ,
\end{align*}
which implies
$$||N^{-1} n_j(\sigma_j)| - 1| \le 2^{-19} \quad , \quad ||N^{-1}n_3'(\sigma_3)|-1| \le 2^{-25} \, , $$
Similarly as in the proof of Prop. \ref{Prop.2.1} this implies
$$
|\tilde{n_j}(\tilde{\sigma_j}) - e_j| \le 2^{-17} \quad , \quad
|\tilde{n_3}'(\tilde{\sigma_3})-e_4| \le 2^{-23} \, .
$$
This implies (II).\\
Proof of (III). Let $\tilde{\sigma_i} , \tilde{\sigma_i}^0 \in \tilde{S_i}$ . We obtain as in the proof of Prop. \ref{Prop.2.1}:
$$
|\tilde{n_i}(\tilde{\sigma_i})-\tilde{n_i}(\tilde{\sigma_i}^0)| \le 3 |N^{-1} n_i(T\tilde{\sigma_i}) - N^{-1} n_i(T\tilde{\sigma_i}^0)|\, .
$$
This implies:
\begin{align*}
\frac{|\tilde{n_i}(\tilde{\sigma_i})-\tilde{n_i}(\tilde{\sigma_i}^0)|}{|\tilde{\sigma_i}-\tilde{\sigma_i}^0|} & \le 3 \|N^{-1}\|\frac{|n_i(T\tilde{\sigma_i}) - n_i(T\tilde{\sigma_i}^0)|}{|\tilde{\sigma_i}-\tilde{\sigma_i}^0|}\\
& \le 3 \|N^{-1}\| \|T\| \frac{|n_i(T\tilde{\sigma_i}) - n_i(T\tilde{\sigma_i}^0)|}{|T\tilde{\sigma_i}-T\tilde{\sigma_i}^0|}\\
& \le 2^{-16}(1-c)  \frac{|n_i(\sigma_i)-n_i(\sigma_i^0)|}{|\sigma_i-\sigma_i^0|} 
 \le 2^{-16}(1-c) \frac{2}{|\xi_i^0|} \le 2^{-7} \le 1
\end{align*}
 by $|\xi_i^0| \ge \frac{N_1}{2N_0} \ge 2^{-8}(1-c)$ for $i=1,2$.  We used that by Remark 1 above $N_0 \le 2^7(1-c)^{-1} N_1$ . This is also true for $i=3$ using that $|\xi_3| \ge \half$ by (\ref{2.18}). As in the proof of Prop. \ref{Prop.2.1} we finally obtain also
$$ \frac{|\tilde{n_3}'(\tilde{\sigma_3}) -\tilde{n_3}'(\tilde{\sigma_3}^0)|}{|\tilde{\sigma_3}-\tilde{\sigma_3}^0|} \le 1$$

Thus the manifolds $\tilde{S_1},\tilde{S_2},\tilde{S_3}$ fulfill the assumption 1.1 of \cite{BH} with parameters $R = 1$ , $b =1$ , $\Theta = \half$ and $\beta=1$ , so that we obtain by \cite{BH}, Theorem 1.3:
$$ \| f_{| \tilde{S_1}} \ast g_{| \tilde{S_2}} \|_{L^2(\tilde{S_3})} \lesssim \|f\|_{L^2(\tilde{S_1})} \|g\|_{L^2(\tilde{S_2})} \, . $$
By application of the linear invertible mapping $T$ we obtain by \cite{BH}, Proposition 1.2 the desired estimate (\ref{2.19}), namely
$$ \|f_{|S_1} \ast g_{|S_2}\|_{L^2(S_{3,\tilde{d}}{\pm})} \lesssim d^{-\half} \|f\|_{L^2(S_1)} \|g\|_{L^2(S_2)} \lesssim A^{\half} \|f\|_{L^2(S_1)} \|g\|_{L^2(S_2)} \, , $$
where $d = |\det(n_{S_1},n_{S_2},n_{S_3},n_{S_3}')| \gtrsim A^{-1}$ . This completes the proof of case (II).\\
{\bf Case (I):} We prepare the proof by two lemmas. The first lemma is an  analogue in 3D of \cite{K}, Lemma 4.6.
\begin{lemma}
	\label{Lemma2.1}
	Let $supp \, g_1 \subset K^-_{N_1,L_1} \cap Q_{j_1}^A$ , $supp \, g_2 \subset K^+_{N_2,L_2} \cap Q_{j_2}^A$  , where $1 \ll N_0 \lesssim N_1 \sim N_2$ , $8 \le A \le M_1 := 2^5 (1-c)^{-\half} \frac{(N_1 N_2)^{\half}}{N_0}$ and $\frac{\pi}{2A} \le \pi - \alpha(j_1,j_2) \le \frac{2 \pi}{A}$. Then the following estimate holds:
	$$ \|\chi_{K^{\pm_0}_{N_0,L_0}} (g_1 \ast g_2)\|_{L^2_{\xi \tau}} \lesssim A^{\half} N_0 L_1^{\half} L_2^{\half} \|g_1\|_{L^2} \|g_2\|_{L^2} \, . $$
	\end{lemma}
\begin{proof}
	As before we may assume that $supp \, g_i \subset  Q_{j_i}^{kA}$ for, say,  $k=4$ . 
	
	Moreover we claim that we may reduce to the case where the space variables fulfill $supp \, g_1 \subset C_{N_0}(\xi') := \{\xi : |\xi-\xi'|\le N_0\}$ for  some $\xi' \in \mathbb{R}^3$ .  To see this we choose a finite covering 
	 $supp \, g_1 \subset \cup_l C_{N_0}(\xi_l') \, , \, |\xi_l'-\xi_{l'}'| \ge N_0 \, \forall l\neq l' \, ,\,  \# l \sim (\frac{N_1}{N_0})^2 $ and similarly $supp \, g_2 \subset \cup_n C_{N_0}(\xi_n'') \, , \, |\xi_n''-\xi_{n'}''| \ge N_0 \, \forall n\neq n' \, ,\,  \# n \sim (\frac{N_2}{N_0})^2 $  . Now fix $l$. If $\xi_1,\xi_1^0 \in C_{N_0}(\xi_l')$ so that $|\xi_1-\xi_1^0| \le 2N_0$ , and 
	 $|\xi - \xi^0| \le 4 N_0$ , which we may asume, we obviously obtain $|\xi_2-\xi_2^0| \le 6N_0$ , so that $|\xi_2-\xi_n''| \le 7N_0$ and $|\xi_2^0-\xi_n''| \le 7N_0$ for some $n$ . Because $|\xi_n''-\xi_{n'}''| \ge N_0$ this means that there exist only finitely many, independent of $l$ and $N_0$, indices $n$ with the property that $\xi_2,\xi_2^0 \in C_{N_0}(\xi_n'')$ and vice versa. By an easy orthogonality argument this implies our claim. 
	 
	 As in the proof of Prop. \ref{Prop.2.2} we have to prove $$ \sup_{(\xi,\tau) \in K^{\pm_0}_{N_0,L_0}} |E(\tau,\xi)| \lesssim A N_0^2 L_1 L_2\, . $$
	 where 
	 \begin{align*}
	  E(\tau,\xi) := \{ (\tau_1,\xi_1) \in Q^{j_1}_{kA}  : \xi_1 \in C_{N_0}(\xi') \, , \, &\langle \tau-\tau_1 -|\xi-\xi_1| \rangle \sim L_1 \,, \langle \tau_1+|\xi_1| \rangle \sim L_2 \\
	  & , \, (\tau-\tau_1,\xi-\xi_1) \in Q^{j_2}_{A} \}\, . 
	  \end{align*}
	 It is immediately clear that for fixed $\xi_1$ we obtain
	 $$ | \{\tau_1 : (\tau_1,\xi_1) \in E(\tau,\xi) \}| \lesssim \min (L_0,L_1) \, . $$ 
	 By rotation we may assume $j_1=0$ . We split $E(\tau,\xi)$ into the parts
	 $E_+(\tau,\xi)$ and $E_-(\tau,\xi)$ , where $|(\xi-\xi_1)_2| \ge |(\xi-\xi_1)_3|$  and $|(\xi-\xi_1)_2| \le |(\xi-\xi_1)_3|$ , respectively. In the first case we estimate
	 \begin{align*}
	 |\partial_2(\tau+|\xi_1| - |\xi-\xi_1|)| & = \left|\frac{(\xi_1)_2}{|\xi_1|} + \frac{(\xi-\xi_1)_2}{|\xi-\xi_1|} \right|\\
	 & \ge  \frac{|(\xi-\xi_1)_2|}{|\xi-\xi_1|} -\frac{|(\xi_1)_2|}{|\xi_1|} \, . 
	 \end{align*}
Now using polar coordinates
	$\frac{\xi-\xi_1}{|\xi-\xi_1|} = (\cos \Theta, \sin \Theta \sin \theta, \sin \Theta \cos \theta)$ with $\pi -\Theta \ge \frac{\pi}{2A}$ we obtain 
	$\frac{|(\xi-\xi_1)_2|}{|\xi-\xi_1|} = |\sin \Theta \sin \theta| \ge \frac{1}{2\sqrt{2}}(\pi - \Theta) \ge \frac{\pi}{4\sqrt{2} A}$ , where we remark that in the case of $E_+$ we have $|\sin \theta| \ge |\cos \theta|$ , thus $|\sin \theta| \ge \frac{1}{\sqrt{2}}$ . Moreover for $(\tau_1,\xi_1) \in Q^0_{4A}$ we have $\frac{|(\xi_1)_2|}{|\xi_1|} =|\frac{(\xi_1)_2}{|\xi_1|} - (1,0,0)_2|\le  \frac{1}{4A}$ . This implies
	$$|\partial_2(\tau+|\xi_1| - |\xi-\xi_1|)| \gtrsim A^{-1} \, .$$
	Combining this with $$|\tau+|\xi_1|-|\xi-\xi_1|| \lesssim |\tau-\tau_1-|\xi-\xi_1| + |\tau_1+|\xi_1||  \lesssim \max(L_0,L_1) $$
	we obtain 
	$$ |\{ (\xi_1)_2 : (\tau_1,\xi_1) \in E_+(\tau,\xi) \}| \lesssim A \max(L_0,L_1) \, . $$
	By the condition $\xi_1 \in C_{N_0}(\xi')$ for some $\xi'$ we obtain
	$$|E_+(\tau,\xi)| \lesssim A N_0^2 \max(L_0,L_1) \min (L_0,L_1) = AN_0^2 L_0 L_1 \, . $$ In the same way we obtain in the case $E_-$ :
	$$|\partial_3(\tau+|\xi_1|-|\xi-\xi_1|)| \gtrsim A^{-1} \, , $$
	which implies the same estimate for $E_-(\tau,\xi)$ . This completes the proof.
	\end{proof}
The second lemma is also essentially given by Kinoshita (\cite{K}, Lemma 4.5).
\begin{lemma}
	\label{Lemma2.2}
	Let $\tau=\tau_1+\tau_2$ , $\xi=\xi_1+\xi_2$ , $ 0 < c < 1$ , $N_1 \ge \frac{1-c}{8(1+c)} N_2$ , and $A \le M_1 := 2^5 (1-c)^{-\half} \frac{(N_1 N_2)^{\half}}{N_0}$ ,
	and assume that we are in Case (I). Then the following estimate holds:
	$$L_{max} := \max(\langle \tau \pm c|\xi|\rangle,\langle \tau_1 - |\xi_1| \rangle,\langle \tau_2 + |\xi_2|\rangle) \gtrsim A^{-\frac{3}{4}} |\xi|\, . $$
	\end{lemma}
\begin{proof}
	We obtain
	\begin{align*}
	L_{max} &\ge |\pm c|\xi|+|\xi_1|-|\xi_2||\\& \ge |\pm c|\xi| +|\xi| \cos \angle(\xi,\xi_1)| - ||\xi_1|-|\xi_2|-|\xi| \cos \angle(\xi,\xi_1)| =: K_1-K_2 \, .
	\end{align*}
	Now we use our assumption (I) and obtain
	\begin{align*}
	K_1 = |\xi| |\cos \theta_0^{\pm} - \cos \angle(\xi,\xi_1)| \ge |\xi| \frac{(1-c)^{\half}}{ 4} |\theta_0^{\pm} - \angle(\xi,\xi_1)| \ge  2^8 (1-c)^{-\half} |\xi| A^{-\frac{3}{4}} \, .
	\end{align*}	
and
\begin{align*}
K_2 & = ||\xi_1|-|\xi-\xi_1|-|\xi| \cos \angle (\xi,\xi_1)|= \left|\frac{2\langle \xi,\xi_1\rangle -|\xi|^2}{|\xi_1|+|\xi-\xi_1|}-|\xi| \cos \angle (\xi,\xi_1)\right| \\
& \le \left| \frac{2\langle \xi,\xi_1\rangle}{|\xi_1|+|\xi-\xi_1|} - |\xi| \cos \angle(\xi,\xi_1) \right| + \frac{|\xi|^2}{|\xi_1|} =\frac{2|\langle \xi,\xi_1\rangle|||\xi_1|-|\xi-\xi_1||}{|\xi_1|||\xi_1|+|\xi-\xi_1||} + \frac{|\xi|^2}{|\xi_1|} \\
& \le \frac{2|\xi|||\xi_1|-|\xi-\xi_1||}{|\xi_1|+|\xi-\xi_1|} +\frac{|\xi|^2}{|\xi_1|} \le \frac{3|\xi|^2}{|\xi_1|} \le 3|\xi| 2^2 \frac{N_0}{N_1}  \\
&
\le 3 |\xi| 2^2 \frac{N_0}{N_1^{\half} N_2^{\half}} (\frac{8(1+c)}{1-c})^{\half}
\le 3 |\xi| 2^2  (\frac{8(1+c)}{1-c})^{\half}               2^5 (1-c)^{-\half} A^{-1} \\
&\le 2^{11} |\xi| (1-c)^{-1} A^{-1} \le 2^{11} |\xi| (1-c)^{-1} A^{-\frac{3}{4}} 2^{-5} (1-c)^{\half}
\le 2^6 |\xi| (1-c)^{-\half} A^{-\frac{3}{4}}
\end{align*}
where we used our assumptions 
$ A^{-1} \ge 2^{-5}(1-c)^{\half} \frac{N_0}{N_1^{\half} N_2^{\half}} $ and $A > 2^{20} (1-c)^{-2}$.
This implies
$$ K_1-K_2 \gtrsim |\xi|A^{-\frac{3}{4}} \, . $$
\end{proof}
Now we proceed with the proof of Prop. \ref{Prop.2.2}. We just proved that in this case $L_{max} \gtrsim A^{-\frac{3}{4}} N_0$ . Let us first consider the case $L_0 \gtrsim A^{-\frac{3}{4}} N_0$ . By Lemma \ref{Lemma2.1} we obtain
		\begin{align*}
	&	|\int f(\xi,\tau) g_1(\xi_1,\tau_1) g_2(\xi_2,\tau_2) d\xi_1 d\xi_2 d\tau_1 d\tau_2| \le \|f\|_{L^2} \|g_1 \ast g_2\|_{L^2} \\
	& \lesssim A^{\half} N_0 L_1^{\half} L_2^{\half} \|f\|_{L^2} \|g_1\|_{L^2} \|g_2\|_{L^2}  \lesssim A^{\frac{7}{8}} N_0^{\half} L_0^{\half} L_1^{\half} L_2^{\half} \|f\|_{L^2} \, . \|g_1\|_{L^2} \|g_2\|_{L^2}\, . 
		\end{align*}
In the case $L_1 \gtrsim 	A^{-\frac{3}{4}} N_0$ we obtain by Prop. \ref{Prop.1.1}:	
	\begin{align*}
&	|\int f(\xi,\tau) g_1(\xi_1,\tau_1) g_2(\xi_2,\tau_2) d\xi_1 d\xi_2 d\tau_1 d\tau_2| \le \|g_1\|_{L^2} \|f \ast g_2\|_{L^2} \\
& \lesssim  N_0 L_0^{\half} L_2^{\half} \|f\|_{L^2} \|g_1\|_{L^2} \|g_2\|_{L^2}  \lesssim A^{\frac{3}{8}} N_0^{\half} L_0^{\half} L_1^{\half} L_2^{\half} \|f\|_{L^2} \, . \|g_1\|_{L^2} \|g_2\|_{L^2}\, . 
\end{align*}
The case $L_2 \gtrsim 	A^{-\frac{3}{4}} N_0$ is treated similarly. This completes the proof in Case (I).\\
{\bf Case (III):} This can be shown similarly as Case (II), even easier. Using the notation above, we obtain this time:
$\# j \sim A^2$ , but this makes no difference, because $A^2 \sim A \sim A^{\half} \sim 1$ . We obtain essentially the same estimates as in case (II) , e.g.
\begin{align*}
|n_1(\sigma_1) - n_1(\sigma_1')| &\le 2^{-37} (1-c)^2 A^{-\half} \frac{N_0}{N_1} \\
&\le 2^{-37} (1-c)^2 2^{-20} (1-c)^2 A^{-\frac{3}{2}} 2^7 (1-c)^{-1} = 2^{-50} (1-c)^3 A^{-\frac{3}{2}}
\end{align*}
by the assumption $A \le 2^{20} (1-c)^{-2}$ and $\frac{N_0}{N_1} \le 2^7 (1-c)^{-1}$ (cf. Remark 1) . Similarly
\begin{align*}
|\langle \sigma_1 - \sigma_1',n_1(\sigma_1') \rangle| \le 2^{-74} (1-c)^4 A^{-1} \frac{N_0}{N_1} \le 2^{-87} (1-c)^5 A^{-2} \, .
\end{align*}
The remainder of the proof is exactly as in case (II).
\end{proof}

Next we consider the case when $\angle(\xi_1,\xi_2)$ is close to $\pi$ , which is much easier.
\begin{prop}
	\label{Prop.2.6}
	Let $\tau=\tau_1+\tau_2$ , $\xi=\xi_1+\xi_2$ and $M_1 \ge 2^5 (1-c)^{-\half} \frac{N_1^{\half} N_2^{\half}}{N_0}$ . Assume $(\xi_1,\tau_1) \in Q^{j_1}_{M_1}$ ,  $(\xi_2,\tau_2) \in Q^{j_2}_{M_1}$ , where  $0 \le \pi - \angle(\xi_1,\xi_2) \le \frac{\pi}{M_1}$ . Then the following estimate holds: 
		$$\max(\langle \tau \pm c|\xi|\rangle,\langle \tau_1 - |\xi_1|\rangle,\langle \tau_2 + |\xi_2|\rangle) \gtrsim |\xi| \, . $$
	\end{prop}
\begin{proof}
	We start with the estimate
	$$\max(\langle \tau \pm c|\xi|\rangle,\langle \tau_1 - |\xi_1|\rangle,\langle \tau_2 + |\xi_2|\rangle) \ge ||\xi_1|-|\xi_2||-c|\xi|\,. $$
	We obtain
	\begin{align*}
	|\xi|^2 &= |\xi_1|^2+|\xi_2|^2 + 2 \langle \xi_1, \xi_2 \rangle = ||\xi_1|-|\xi_2||^2 + 2|\xi_1||\xi_2| (1+ \cos \angle(\xi_1,\xi_2)) \\
	& \le ||\xi_1|-|\xi_2||^2 +2|\xi_1||\xi_2| |\pi - \angle(\xi_1,\xi_2)|^2 \\
	&\le ||\xi_1|-|\xi_2||^2 + 2|\xi_1||\xi_2| \pi^2 2^{-10} (1-c) \frac{N_0^2}{N_1 N_2} \\
	& \le  ||\xi_1|-|\xi_2||^2 + 2|\xi_1||\xi_2| \pi^2 2^{-6} (1-c) \frac{|\xi|^2}{|\xi_1| |\xi_2|} \\
	& \le ||\xi_1|-|\xi_2||^2 + \frac{1-c}{2} |\xi|^2\, .
	\end{align*}
	This implies
	$$ \frac{1+c}{2} |\xi|^2\le ||\xi_1|-|\xi_2||^2 \, , $$
	so that
	$$\max(\langle \tau \pm c|\xi|\rangle,\langle \tau_1 - |\xi_1|\rangle,\langle \tau_2 - |\xi_2|\rangle) \ge (\sqrt{\frac{1+c}{2}}-c) |\xi| \ge \frac{1-c}{4} |\xi| \, . $$
	\end{proof}

We combine Prop. \ref{Prop.2.5} and Prop. \ref{Prop.2.6} to prove (\ref{2.0}) and (\ref{2.0'}) .\begin{prop}
	\label{Prop.2.7}
	Assume $s >-\half$ and $1 \le N_0 \lesssim N_1 \sim N_2$ , $M_1 := 2^5 (1-c)^{-\half} \frac{(N_1 N_2)^{\half}}{N_0} $ , $ L_{max} \ll N_1$ . Then the following estimates apply:
	\begin{align*}
&\sum_{N_1} \sum_{1 \le N_0 \lesssim N_1 \sim N_2} \sum_{L_0,L_1,L_2 \le N_1} \sum_{\stackrel{-M_1 \le j_1,j_2\le M_1-1}{0 \le \pi - \alpha(j_1,j_2) \le \pi M_1^{-1}}}
 |I_1(f,{g_1}_{|Q^{j_1}_{M_1}},{g_2}_{|Q^{j_2}_{M_1}})| 
 \\&\lesssim \|f\|_{X^{s,b}_{\pm,c}} \|g_1\|_{X^{s,b}_-} \|g_2\|_{X^{-s,1-b-}_+} \\
&\sum_{N_1} \sum_{1 \le N_0 \lesssim N_1 \sim N_2} \sum_{L_0,L_1,L_2 \le N_1} \sum_{\stackrel{-M_1 \le j_1,j_2\le M_1-1}{0 \le \pi - \alpha(j_1,j_2) \le \pi M_1^{-1}}}
|I_2(f,{g_1}_{|Q^{j_1}_{M_1}},{g_2}_{|Q^{j_2}_{M_1}})| 
\\&\lesssim \|f\|_{X^{-s,1-b-}_{\pm,c}} \|g_1\|_{X^{s,b}_-} \|g_2\|_{X^{s,b}_+} \, , 
 \end{align*}
 where 
 $$ I_1 := N_1^{-1} \int (P_{K^{\pm,c}_{N_0,L_0}} f) (P_{K_{N_1,L_1}^-}g_1)(P_{K_{N_2,L_2}^+} g_2) dt dx \, , $$
 and
  $$ I_2 := N_0 N_1^{-1}N_2^{-1} \int (P_{K^{\pm,c}_{N_0,L_0}} f) (P_{K_{N_1,L_1}^-}g_1)(P_{K_{N_2,L_2}^+} g_2) dt dx \, . $$
\end{prop}
\begin{proof}
	Using $L_{max} \lesssim N_1$ and $b=\half+$ it suffices to prove for  $supp \, f \subset K^{\pm,c}_{N_0,L_0}$ , $supp \, g_1 \subset  K^-_{N_1,L_1}$ , $supp \, g_2 \subset K^+_{N_2,L_2}$ the estimate
	\begin{align*}
	 |I_1(f,{g_1}_{|Q^{j_1}_{M_1}},{g_2}_{|Q^{j_2}_{M_1}})| \lesssim N_0^s N_1^{s-} N_2^{-s} (L_0 L_1 L_2)^{\half} \|f\|_{L^2} \|{g_1}_{|Q^{j_1}_{M_1}}\|_{L^2} \|{g_2}_{|Q^{j_2}_{M_1}}\|_{L^2} \, . 
	 \end{align*}
	By Prop. \ref{Prop.2.6} we obtain $L_{max} \gtrsim N_0$ . Let $L_0 \gtrsim N_0$ (the other cases being similar). We obtain by Prop. \ref{Prop.1.1}:
	\begin{align*}
	 &|I_1(f,{g_1}_{|Q^{j_1}_{M_1}},{g_2}_{|Q^{j_2}_{M_1}})| \lesssim N_1^{-1} \|\chi_{K^{\pm,c}_{N_0,L_0}} ({g_1}_{|Q^{j_1}_{M_1}}{g_2}_{|Q^{j_2}_{M_1}})\|_{L^2} \|f\|_{L^2} \\
	 & \lesssim N_1^{-1} N_0^{\half} N_1^{\half} L_1^{\half} L_2^{\half} \frac{L_0^{\half}}{N_0^{\half}} \|{g_1}_{|Q^{j_1}_{M_1}}\|_{L^2} \|{g_2}_{|Q^{j_2}_{M_1}}\|_{L^2} \|f\|_{L^2} \, .
	\end{align*}
	It remains to prove $N_1^{-\half} \lesssim N_0^s N_1^{0-}$ , which applies for $s > -\half$ . Moreover we need
	$$ |I_2(f,{g_1}_{|Q^{j_1}_{M_1}},{g_2}_{|Q^{j_2}_{M_1}})| \lesssim N_0^{-s} N_1^{s-} N_2^{s} (L_0 L_1 L_2)^{\half} \|f\|_{L^2} \|{g_1}_{|Q^{j_1}_{M_1}}\|_{L^2} \|{g_2}_{|Q^{j_2}_{M_1}}\|_{L^2} \, . $$ 
	By Prop. \ref{Prop.2.6} we obtain similarly
	\begin{align*}
	&|I_2(f,{g_1}_{|Q^{j_1}_{M_1}},{g_2}_{|Q^{j_2}_{M_1}})| \lesssim N_0 N_1^{-1}  N_2^{-1} \|\chi_{K^{\pm,c}_{N_0,L_0}} ({g_1}_{|Q^{j_1}_{M_1}},{g_2}_{|Q^{j_2}_{M_1}})\|_{L^2} \|f\|_{L^2} \\
	& \lesssim N_0 N_1^{-1} N_2^{-1} N_0^{\half} N_1^{\half} L_1^{\half} L_2^{\half} \frac{L_0^{\half}}{N_0^{\half}} \|{g_1}_{|Q^{j_1}_{M_1}}\|_{L^2} \|{g_2}_{|Q^{j_2}_{M_1}}\|_{L^2} \|f\|_{L^2} \, .
	\end{align*}
	This implies the desired estimate, because
	$$ N_0 N_1^{-\frac{3}{2}} \lesssim N_0^{-s} N_1^{2s-} \, \Longleftrightarrow \, N_0^{1+s} \lesssim N_1^{\frac{3}{2}+2s-} $$ if $ s >- \half$ .
	\end{proof}

\begin{prop}
Assume $s >-\half$ and $1 \le N_0 \lesssim N_1 \sim N_2$ , $N_1 \ge 2^{-4} (1-c) N_2$ and $ N_2 \ge 2^{-4} (1-c) N_1$ , $M_1$ , $I_1$ , $I_2$ as in Prop.  \ref{Prop.2.7} , $ L_{max} \ll N_1$ . Then the following estimates apply:
\begin{align*}
&\sum_{N_1} \sum_{\stackrel{1 \le N_0 \lesssim N_1 \sim N_2}{N_1 \ge \frac{1-c}{2^4} N_2 , N_2 \ge \frac{1-c}{2^4} N_1}} \sum_{L_0,L_1,L_2 \le N_1} \sum_{1 \ll A \le M_1} \sum_{\stackrel{-A \le j_1,j_2\le A-1}{\frac{\pi}{2A} \le |\pi - \alpha(j_1,j_2)| \le \frac{2\pi}{A}}} \hspace{-1.7em}
|I_1(f,{g_1}_{|Q^{j_1}_{A}},{g_2}_{|Q^{j_2}_{A}})| 
\\& \quad\lesssim \|f\|_{X^{s,b}_{\pm,c}} \|g_1\|_{X^{s,b}_-} \|g_2\|_{X^{-s,1-b-}_+} \\
&\sum_{N_1} \sum_{\stackrel{1 \le N_0 \lesssim N_1 \sim N_2}{N_1 \ge \frac{1-c}{2^4} N_2 , N_2 \ge \frac{1-c}{2^4} N_1}} \sum_{L_0,L_1,L_2 \le N_1}  \sum_{1 \ll A \le M_1} \sum_{\stackrel{-A \le j_1,j_2\le A-1}{\frac{\pi}{2A} \le |\pi - \alpha(j_1,j_2)| \le \frac{2\pi}{A}}} \hspace{-1.7em}
|I_2(f,{g_1}_{|Q^{j_1}_{A}},{g_2}_{|Q^{j_2}_{A}})| 
\\& \quad \lesssim \|f\|_{X^{-s,1-b-}_{\pm,c}} \|g_1\|_{X^{s,b}_-} \|g_2\|_{X^{s,b}_+} \, . 
\end{align*}	
\end{prop}
\begin{proof}
We only treat the case $1 \ll N_0$ . Otherwise the result is easily obtained by H\"older's inequality.
	Using $L_{max} \lesssim N_1$ and $b=\half+$ it suffices to prove for  $supp \, f \subset K^{\pm,c}_{N_0,L_0}$ , $supp \, g_1 \subset  K^-_{N_1,L_1}$ , $supp \, g_2 \subset K^+_{N_2,L_2}$ the estimate
\begin{align*}
 \sum_{1 \ll A \le M_1} \sum_{\stackrel{-A \le j_1,j_2\le A-1}{\frac{\pi}{2A} \le |\pi - \alpha(j_1,j_2)| \le \frac{2\pi}{A}}} &        |I_1(f,{g_1}_{|Q^{j_1}_A},{g_2}_{|Q^{j_2}_A})| \\
 &\lesssim N_0^s N_1^{s-} N_2^{-s} (L_0 L_1 L_2)^{\half} \|f\|_{L^2} \|{g_1}\|_{L^2} \|{g_2}\|_{L^2} \, . 
 \end{align*}
By Prop. \ref{Prop.2.5} we obtain:
\begin{align*}
&\sum_{1 \ll A \le M_1} \sum_{\stackrel{-A \le j_1,j_2\le A-1}{\frac{\pi}{2A} \le |\pi - \alpha(j_1,j_2)| \le \frac{2\pi}{A}}}  |I_1(f,{g_1}_{|Q^{j_1}_A},{g_2}_{|Q^{j_2}_A})| \\
& \lesssim \sum_{1 \ll A \le M_1} \sum_{\stackrel{-A \le j_1,j_2\le A-1}{\frac{\pi}{2A} \le |\pi - \alpha(j_1,j_2)| \le \frac{2\pi}{A}}}  N_1^{-1} N_0^{\half}  (L_0 L_1 L_2)^{\half} A^{\frac{7}{8}} \|{g_1}_{|Q^{j_1}_A}\|_{L^2} \|{g_2}_{|Q^{j_2}_A}\|_{L^2} \|f\|_{L^2} \\
& \lesssim  N_1^{-1} N_0^{\half}  N_1^{\frac{7}{8}} N_0^{-\frac{7}{8}}(L_0 L_1 L_2)^{\half}  \|{g_1}\|_{L^2} \|{g_2}\|_{L^2} \|f\|_{L^2} 
 \end{align*}
by $A \le M_1 \sim N_1 N_0^{-1}$ .
It remains to prove $N_1^{-1} N_0^{\half} N_1^{\frac{7}{8}} N_0^{-\frac{7}{8}}\lesssim N_0^s N_1^{0-} \, \Leftrightarrow \, N_0^{-s-\frac{3}{8}} \lesssim N_1^{\frac{1}{8}-}$ , which applies for $s > -\half$ . Moreover we need
\begin{align*}
 \sum_{1 \ll A \le M_1} \sum_{\stackrel{-A \le j_1,j_2\le A-1}{\frac{\pi}{2A} \le |\pi - \alpha(j_1,j_2)| \le \frac{2\pi}{A}}} &|I_2(f,{g_1}_{|Q^{j_1}_A},{g_2}_{|Q^{j_2}_A})| \\
 &\lesssim N_0^{-s} N_1^{s-} N_2^{s} (L_0 L_1 L_2)^{\half} \|f\|_{L^2} \|{g_1}\|_{L^2} \|{g_2}\|_{L^2} \, . 
\end{align*} 
By Prop. \ref{Prop.2.5} we obtain 
\begin{align*}
&\sum_{1 \ll A \le M_1} \sum_{\stackrel{-A \le j_1,j_2\le A-1}{\frac{\pi}{2A} \le |\pi - \alpha(j_1,j_2)| \le \frac{2\pi}{A}}}|I_2(f,{g_1}_{|Q^{j_1}_A},{g_2}_{|Q^{j_2}_A})|  \\
& \lesssim 
 \sum_{1 \ll A \le M_1} \sum_{\stackrel{-A \le j_1,j_2\le A-1}{\frac{\pi}{2A} \le |\pi - \alpha(j_1,j_2)| \le \frac{2\pi}{A}}}N_0 N_1^{-1} N_2^{-1} N_0^{\half}  (L_0 L_1 L_2)^{\half} A^{\frac{7}{8}} \\ & \hspace{10em} \cdot
 \|{g_1}_{|Q^{j_1}_A}\|_{L^2} \|{g_2}_{|Q^{j_2}_A}\|_{L^2} \|f\|_{L^2} \\
 & \lesssim N_0 N_1^{-1} N_2^{-1} N_0^{\half} N_1^{\frac{7}{8}} N_0^{-\frac{7}{8}} (L_0 L_1 L_2)^{\half}   \|{g_1}\|_{L^2} \|{g_2}\|_{L^2} \|f\|_{L^2}
\end{align*}
by $A \le M_1 \sim N_1 N_0^{-1}$ . This implies the desired estimate, because
$$ N_0^{\frac{3}{2}} N_1^{-2} N_1^{\frac{7}{8}} N_0^{-\frac{7}{8}} \lesssim N_0^{-s} N_1^{2s-} \, \Longleftrightarrow \, N_0^{s+\frac{5}{8}} \lesssim N_1^{\frac{9}{8}+2s-} $$ if $ s >- \half$ .
\end{proof}
\begin{proof}[Proof of Proposition \ref{Theorem2.1}]
Summarizing the results in chapter 1 and chapter 2 we have proven the estimates (\ref{1.5}) and (\ref{1.6}) whenever $ s >-\half$ , and thus also Proposition \ref{Theorem2.1}.
\end{proof}
\begin{proof}[Proof of Theorem \ref{Theorem0.1}]
	By well-known arguments (cf. e.g. \cite{GTV}) the bilinear estimates in Proposition \ref{Theorem2.1} directly imply Theorem \ref{Theorem0.1}. Corollary \ref{Cor.1} and the subsequent remark are also standard consequences of the used iteration argument in Bourgain-Klainerman-Machedon spaces (cf. also e.g. \cite{GTV}).
	\end{proof}

\end{document}